\documentclass[11pt,a4paper]{article}
\usepackage{amssymb,amsmath}
\usepackage{tikz}
\usetikzlibrary{arrows,patterns}

\usepackage{amssymb}
\newtheorem{theorem}{Theorem}[section]
\newtheorem{corollary}[theorem]{Corollary}
\newtheorem{definition}[theorem]{Definition}
\newtheorem{lemma}[theorem]{Lemma}
\newtheorem{proposition}[theorem]{Proposition}
\newtheorem{remark}[theorem]{Remark}

\newtheorem{example}[theorem]{Example}

\newenvironment{proof}{\begin{trivlist}\item[]{\it Proof.}}
{\hfill$\square$\end{trivlist}}

\newenvironment{proofofprop}[1]{\noindent{\it Proof of Proposition#1}}{\hfill$\square$\\\mbox{}}

\def\quiver{Q} 
\def\polytope{\nabla}
\def\floweqs{\mathcal{F}}
\def\ideal{I}
\def\rep{{\mathrm{Rep}}}
\def\moduli{{\mathcal{M}}}
\def\coord{{\mathcal{O}}}
\def\proj{{\mathrm{Proj}}}
\def\affspan{{\mathrm{AffSpan}}}
\def\skeleton{{\mathcal{S}}}
\def\sklist{{\mathcal{L}}}
\def\reducedquivers{\mathcal{R}}

\def\subforest{T}
\def\subtree{T}
\def\support{\mathrm{supp}}
\def\homcoord{{\mathcal{A}}}
\def\semigr{S}

\def\cocoa{{\hbox{\rm C\kern-.13em o\kern-.07em C\kern-.13em o\kern-.15em A}}}
\def\mc{{\mathbb{C}}}
\def\mz{{\mathbb{Z}}}
\def\mn{{\mathbb{N}}}

\def\mr{{\mathbb{R}}}

\begin{document}

\title{On the equations and classification  of toric quiver varieties}
\author{{M. Domokos and D. Jo\'o}\thanks{Partially supported by OTKA  NK81203 and K101515. }}
%The results of this paper will be included in the forthcoming PhD thesis of the second named author at the Central European University. }}

\date{}
\maketitle 
{\small \begin{center} 
R\'enyi Institute of Mathematics, Hungarian Academy of Sciences,\\
Re\'altanoda u. 13-15, 1053 Budapest, Hungary \\
Email: domokos.matyas@renyi.mta.hu \quad joo.daniel@renyi.mta.hu 
\end{center}
}

%\date{July 17, 2013}                                           % Activate to display a given date or no date

\begin{abstract} 
Toric quiver varieties (moduli spaces of quiver representations) are studied. Given a quiver and a weight there is an associated quasiprojective toric variety together with a canonical embedding into projective space. It is shown that for a quiver with no oriented cycles the homogeneous ideal of this embedded projective variety is generated by elements of degree at most $3$. 
In each fixed dimension $d$ up to isomorphism there are only finitely many $d$-dimensional toric quiver varieties. A procedure for their classification is outlined. 
\end{abstract}



%%%%%%%%%%%%%%%%%%%%%%%%%%%%%%%%%%%%%%%%%

\section{Introduction}\label{sec:intro}

Geometric invariant theory was applied by King \cite{king} to introduce certain moduli spaces of representations of quivers. 
In the special case when the dimension vector takes value $1$ on each vertex of the quiver ({\it thin representations}), these moduli spaces are quasi-projective toric varieties; following \cite{altmann-hille} we call them {\it toric quiver varieties}. 
Toric quiver varieties were studied by Hille \cite{hille:chemnitz}, \cite{hille:canada}, \cite{hille:laa}, Altmann and Hille \cite{altmann-hille}, Altmann and van Straten \cite{altmann-straten}. Further motivation is provided by Craw and Smith \cite{craw-smith}, who showed that every projective toric variety  is the fine moduli space for stable thin representations of an appropriate quiver with relations. Another application was introduced  very recently by Carroll, Chindris and Lin \cite{carroll-chindris-lin}. 
From a different perspective, the projective toric quiver varieties are nothing but the  toric varieties associated to flow polytopes. 
Taking this point of departure, Lenz \cite{lenz} investigated toric ideals associated to flow polytopes. 
These are the homogeneous ideals of the projective toric variety associated to a flow polytope, canonically embedded into projective space.

Given a quiver (a finite directed graph) and a weight (an integer valued function on the set of vertices) there is an associated normal lattice polyhedron yielding a quasiprojective toric variety with a canonical embedding into projective space. This variety is projective if and only if the quiver has no oriented cycles. We show in Theorem~\ref{thm:degreethree} that the homogeneous ideal of this embedded projective variety is generated by elements of degree at most $3$. This is deduced from a recent result of 
Yamaguchi, Ogawa and Takemura \cite{yamaguchi-ogawa-takemura}, for which we give a simplified proof.   

It follows from work of Altmann and van Straten \cite{altmann-straten} and  Altmann,  Nill,  Schwentner and  Wiercinska \cite{altmann-nill-schwentner-wiercinska} that  
for each positive integer $d$ up to isomorphism there are only finitely many toric quiver varieties (although up to integral-affine equivalence  there are infinitely many $d$-dimensional quiver polyhedra). We make this notable observation explicit and provide a self-contained treatment yielding some refinements. 
Theorem~\ref{thm:directprod} asserts that any toric quiver variety is the product of prime (cf. Definition~\ref{def:prime}) toric quiver varieties, and this deomposition can be read off from the combinatorial structure of the quiver. 
Moreover, by Theorem~\ref{thm:3regular} any prime (cf. Definition~\ref{def:prime}) 
$d$-dimensional $(d>1)$ projective toric quiver variety can be obtained from a bipartite quiver with $5(d-1)$ vertices and $6(d-1)$ arrows, whose skeleton (cf. Definition~\ref{def:skeleton}) is $3$-regular.

A toric variety associated to a lattice polyhedron is covered by affine open toric subvarieties corresponding to the vertices of the polyhedron. 
In the case of quiver polyhedra the affine toric varieties arising that way are exactly the affine toric quiver varieties by our Theorem~\ref{thm:vertices} and Theorem~\ref{thm:compactification}.  According to Theorem~\ref{thm:compactification} any toric quiver variety can be obtained as the union in a projective toric quiver variety of the affine open subsets corresponding to a set of vertices of the quiver polytope. 

The paper is organized as follows. 
In Section~\ref{sec:flowpolytopes}  we review flow polytopes, quiver polytopes, quiver polyhedra, and their interrelations. In Section~\ref{sec:repsofquivers} we recall  moduli spaces of representations of quivers, including a very explicit realization of a toric quiver variety in Proposition~\ref{prop:realizationofmoduli}.  In Section~\ref{sec:reductions} we collect reduction steps for quiver--weight pairs that preserve the associated quiver polyhedron, and can be used to replace a quiver by another one which is simpler or smaller in  certain sense. These are used to derive the results concerning 
the classification of  toric quiver varieties. 
As an illustration the classification of $2$-dimensional toric quiver varieties  is recovered in Section~\ref{sec:2-dimensional}. 
Section~\ref{sec:affine} clarifies the interrelation of affine versus projective toric quiver varieties. 
Section~\ref{sec:semigroupalgebras} contains some generalities on presentations of semigroup algebras, from which we obtain Corollary~\ref{cor:joo} that provides the technical framework for the proof in Section~\ref{sec:equations} of Theorem~\ref{thm:degreethree} about the equations for the natural embedding of a toric quiver variety into projective space. For sake of completeness of the picture we show in Section~\ref{sec:japanok} how the main result of \cite{yamaguchi-ogawa-takemura} can be derived from the special case Proposition~\ref{prop:theta1} used in the proof of Theorem~\ref{thm:degreethree}.  
We also point out in Theorem~\ref{thm:affinebound} that the ideal of relations among the minimal generators of the coordinate ring of a $d$-dimensional  affine toric quiver variety is generated in degree at most $d-1$, and this bound is sharp.

%%%%%%%%%%%%%%%%%%%%%%%%%%%%%%%%%%%%%%%%%%%%%

\section{Flow polytopes and their toric varieties} \label{sec:flowpolytopes}

 By a {\it polyhedron} we mean the intersection of finitely many closed half-spaces in $\mr^n$, and by a {\it polytope} we mean a bounded polyhedron, or equivalently, the convex hull of a finite subset in $\mr^n$ (this conforms the usage of these terms in  \cite{cox-little-schenck}). 
A {\it quiver}  is a finite directed graph $\quiver$ with vertex set $\quiver_0$ and arrow set $\quiver_1$. Multiple arrows, oriented cycles, loops are all allowed. 
For an arrow $a\in \quiver_1$ denote by 
$a^-$ its starting vertex and by $a^+$ its terminating vertex. Given an integral vector $\theta\in\mz^{\quiver_0}$ and non-negative integral vectors 
$\mathbf{l},\mathbf{u}\in\mn_0^{\quiver_1}$ consider the polytope 
\[\polytope=\polytope(Q,\theta,\mathbf{l},\mathbf{u})=\{x\in\mr^{\quiver_1}\mid \mathbf{l}\le x \le\mathbf{u},\quad \forall v\in\quiver_0:\quad 
\theta(v)=\sum_{a^+=v}x(a)-\sum_{a^-=v}x(a)\}.\] 
This is called a {\it flow polytope}. According to  the generalized Birkhoff-von Neumann Theorem $\polytope$ is a {\it lattice polytope} in $\mr^{\quiver_1}$, that is, its vertices belong to the lattice $\mz^{\quiver_1}\subset\mr^{\quiver_1}$ (see for example Theorem 13.11 in \cite{schrijver}).  
Denote by $X_{\polytope}$ the {\it projective toric variety} associated to $\polytope$ (cf. Definition 2.3.14 in \cite{cox-little-schenck}). 
The polytope $\polytope$ is {\it normal} (see Theorem  13.14 in \cite{schrijver}). 
 It follows that the abstract variety $X_{\polytope}$ can be identified with the Zariski-closure of the image of the map 
 \begin{equation}\label{eq:X_P} (\mc^{\times})^{\quiver_1}\to \mathbb{P}^{d-1},\quad t \mapsto (t^{m_1}:\dots:t^{m_d})\end{equation}
 where $\{m_1,\dots,m_d\}=\polytope \cap \mz^{\quiver_1}$, and for $t$ in the torus $(\mc^{\times})^{\quiver_1}$ and $m\in\mz^{\quiver_1}$ we write 
 $t^m:=\prod_{a\in\quiver_1}t(a)^{m(a)}$. From now on $X_{\polytope}$ will stand for this particular embedding in projective space of our variety, and we denote 
 by $\ideal(X_{\polytope})$ the corresponding vanishing ideal, so $\ideal(X_{\polytope})$ is a homogeneous ideal in $\mc[x_1,\dots,x_d]$ generated by binomials.  Normality of $\polytope$ implies that $X_{\polytope}$ is {\it projectively normal}, that is, its affine cone in $\mc^d$ is normal. 
We shall also use the notation 
\[\polytope(\quiver,\theta)=\{x\in\mr^{\quiver_1}\mid \mathbf{0}\le x,\quad \forall v\in\quiver_0:\quad 
\theta(v)=\sum_{a^+=v}x(a)-\sum_{a^-=v}x(a)\}.\] 
We shall call this  a {\it quiver polyhedron}. 
When $\quiver$ has no oriented cycles, then for  ${\mathbf{u}}$ large enough we have 
$\polytope(\quiver,\theta)=\polytope(\quiver,\theta,\mathbf{0},\mathbf{u})$, so $\polytope(\quiver,\theta)$ is a polytope; these polytopes will be called {\it quiver polytopes}. 
\begin{definition}\label{def:isomorphicpolytopes}
{\rm The lattice polyhedra  $\polytope_i\subset V_i$ with lattice $M_i\subset V_i$ $(i=1,2)$ are {\it integral-affinely equivalent} 
if there exists an affine linear isomorphism  $\varphi:\affspan(\polytope_1)\to\affspan(\polytope_2)$ of affine subspaces 
with the following properties: 
\begin{itemize}
\item[(i)] $\varphi$ maps $\affspan(\polytope_1)\cap M_1$ onto $\affspan(\polytope_2)\cap M_2$; 
\item[(ii)] $\varphi$ maps  $\polytope_1$ \ onto $\polytope_2$. 
\end{itemize} }
\end{definition} 

The phrase `integral-affinely equivalent' was chosen in accordance with \cite{bruns-gubeladze} (though in \cite{bruns-gubeladze} full dimensional lattice polytopes are considered).  
Obviously, if $\polytope_1$ and $\polytope_2$ are integral-affinely equivalent lattice polytopes, then the associated projective toric varieties 
$X_{\polytope_1}$ and $X_{\polytope_2}$ are isomorphic (and in fact they can be identified via their embeddings into projective space given by the $\polytope_i$). 
As we shall point out in Proposition~\ref{prop:flowpolytope-quiverpolytope} below, up to integral-affine equivalence, the class of flow polytopes coincides with the class of quiver polytopes, so the class of quiver polyhedra is the most general among the above classes.

\begin{proposition}\label{prop:flowpolytope-quiverpolytope} 
For any flow polytope $\polytope(\quiver, \theta,{\mathbf{l}},{\mathbf{u}})$ there exists a quiver $\quiver'$ with no oriented cycles and a weight $\theta'\in\mz^{\quiver'_1}$ such that the polytopes $\polytope(\quiver, \theta,{\mathbf{l}},{\mathbf{u}})$ and $\polytope(\quiver',\theta')$ are integral-affinely equivalent. 
\end{proposition} 

\begin{proof}  Note that $x\in\mr^{\quiver_1}$ belongs to $\polytope(\quiver,\theta,\mathbf{l},\mathbf{u})$ if and only if 
$x-\mathbf{l}$ belongs to $\polytope(Q,\theta',\mathbf{0},\mathbf{u}-\mathbf{l})$ where 
$\theta'$ is the weight given by $\theta'(v)=\theta(v)-\sum_{a^+=v}\mathbf{l}(a)+\sum_{a^-=v}\mathbf{l}(a)$. 
Consequently $X_{\polytope(\quiver,\theta,\mathbf{l},\mathbf{u})}=X_{\polytope(\quiver,\theta',\mathbf{0},\mathbf{u}-\mathbf{l})}$.
Therefore it is sufficient to deal with the flow polytopes $\polytope(\quiver,\theta,\mathbf{0},\mathbf{u})$. Define a new quiver $Q'$ as follows: add to the vertex set of $Q$ two new vertices $v_a$, $w_a$ for each $a\in\quiver_1$, and replace the arrow $a\in Q_1$ by three arrows $a_1,a_2,a_3$, where $a_1$ goes from $a^-$ to $v_a$, 
$a_2$ goes from $w_a$ to $v_a$, and $a_3$ goes from  $w_a$ to $a^+$. Let $\theta'\in\mz^{\quiver'_0}$ be the weight with $\theta'(v_a)=\mathbf{u}(a)=-\theta(w_a)$ for all $a\in\quiver_1$ and $\theta'(v)=\theta(v)$ for all $v\in \quiver_0$. Consider the linear map $\varphi:\mr^{\quiver_1}\to\mr^{\quiver'_1}$, $x\mapsto y$, where 
$y(a_1):=x(a)$, $y(a_3):=x(a)$, and $y(a_2)=\mathbf{u}(a)-x(a)$ for all $a\in\quiver_1$. It is straightforward to check that $\varphi$ is an affine linear transformation that restricts to an isomorphism 
$\affspan(\polytope(\quiver,\theta,\mathbf{0},\mathbf{u}))\to \affspan(\polytope(\quiver',\theta'))$ with the properties (i) and (ii) in Definition~\ref{def:isomorphicpolytopes}. 
\end{proof}

%%%%%%%%%%%%%%%%%%%%%%%%%%%%%%%%%%%%

\section{Moduli spaces of quiver representations} \label{sec:repsofquivers} 

A {\it representation} $R$ of $\quiver$ assigns a finite dimensional $\mc$-vector space $R(v)$ to each vertex $v\in\quiver_0$ and a linear map 
$R(a):R(a^-)\to R(a^+)$ to each arrow $a\in \quiver_1$. A morphism between representations $R$ and $R'$ consists of a collection of linear maps 
$L(v):R(v)\mapsto R'(v)$ satisfying $R'(a)\circ L(a^-)=L(a^+)\circ R(a)$ for all $a\in\quiver_1$. The {\it dimension vector} of $R$ is $(\dim_{\mc}(R(v))\mid v\in \quiver_0)\in\mn^{\quiver_0}$. For a fixed dimension vector $\alpha\in\mn^{\quiver_0}$, 
\[\rep(\quiver,\alpha):=\bigoplus_{a\in \quiver_1}\hom_{\mc}(\mc^{\alpha(a^-)},\mc^{\alpha(a^+)})\] 
is the {\it space of $\alpha$-dimensional representations} of $\quiver$. The product of general linear groups 
$GL(\alpha):=\prod_{v\in\quiver_0}GL_{\alpha(v)}(\mc)$ acts linearly on $\rep(\quiver,\alpha)$ via 
\[g\cdot R:=(g(a^+)R(a)g(a^-)^{-1}\mid a\in \quiver_1)\quad (g\in GL(\alpha), R\in\rep(\quiver,\alpha)).\] 
The $GL(\alpha)$-orbits in $\rep(\quiver,\alpha)$ are in a natural bijection with the isomorphism classes of $\alpha$-dimensional representations of $\quiver$. 
Given a {\it weight} $\theta\in\mz^{\quiver_0}$, a representation $R$ of $\quiver$ is called {\it $\theta$-semi-stable} if $\sum_{v\in\quiver_0}\theta(v)\dim_{\mc}(R(v))=0$ and 
$\sum_{v\in\quiver_0}\theta(v)\dim_{\mc}(R'(v))\ge0$ for all subrepresentations $R'$ of $R$. The points in $\rep(\quiver,\alpha)$ corresponding to $\theta$-semi-stable representations constitute a Zariski open subset $\rep(\quiver,\alpha)^{\theta-ss}$ in the representation space, and in \cite{king}  Geometric Invariant Theory (cf. \cite{newstead}) is applied to  
define a variety 
$\moduli(\quiver,\alpha,\theta)$ and a morphism 
\begin{equation}\label{eq:quotient} \pi:\rep(\quiver,\alpha)^{\theta-ss}\to\moduli(\quiver,\alpha,\theta)\end{equation} 
which is a {\it coarse moduli space} for families of $\theta$-semistable $\alpha$-dimensional representations of $\quiver$ up to S-equivalence. 
A polynomial function $f$ on $\rep(\quiver,\alpha)$ is a {\it relative invariant of weight} $\theta$ if 
$f(g\cdot R)=(\prod_{v\in \quiver_0}\det(g(v))^{\theta(v)})f(R)$ holds for all $g\in GL(\alpha)$ and $R\in\rep(\quiver,\alpha)$. 
The relative invariants of weight $\theta$ constitute a subspace $\coord(\rep(\quiver,\alpha))_{\theta}$ in the coordinate ring $\coord(\rep(\quiver,\alpha))$ of the affine space 
$\rep(\quiver,\alpha)$.  In fact  $\coord(\rep(\quiver,\alpha))_{\theta}$ is a finitely generated module over the algebra $\coord(\rep(\quiver,\alpha))^{GL(\alpha)}$ of polynomial 
$GL(\alpha)$-invariants on $\rep(\quiver,\alpha)$ (generators of this latter algebra are described in \cite{lebruyn-procesi}). 
Now a quasiprojective variety $\moduli(\quiver,\alpha,\theta)$ is defined as the {\it projective spectrum}  
\[\moduli(\quiver,\alpha,\theta)=\proj(\bigoplus_{n=0}^{\infty}\coord(\rep(\quiver,\alpha))_{n\theta})\]  
of the graded algebra $\bigoplus_{n=0}^{\infty}\coord(\rep(\quiver,\alpha))_{n\theta}$.  
A notable special case is that of the zero weight. Then the moduli space  $\moduli(\quiver,\alpha,0)$ is the affine variety whose coordinate ring is the subalgebra of $GL(\alpha)$-invariants in $\coord(\rep(\quiver,\alpha))$. This was studied in \cite{lebruyn-procesi} before the introduction of the case of general weights in \cite{king}. 
Its points are in a natural bijection with the isomorphism classes of semisimple representations of $\quiver$ with dimension vector $\alpha$. For a quiver with no oriented cycles, $\moduli(\quiver,\alpha,0)$ is just a point, and it is more interesting for quivers containing oriented cycles.

Let us turn to the special case when $\alpha(v)=1$ for all $v\in\quiver_0$; we simply write $\rep(\quiver)$ and $\moduli(\quiver,\theta)$  instead of $\rep(\quiver,\alpha)$ and 
$\moduli(\quiver,\alpha,\theta)$. 
When $\rep(\quiver)^{\theta-ss}$ is non-empty, $\moduli(\quiver,\theta)$ is a quasiprojective toric variety with torus 
$\pi(\{x\in\rep(\quiver)\mid x(a)\neq 0\  \forall a\in\quiver_1\}) = \pi((\mc^\times)^{\quiver_1})$. On the other hand it is well known (see Proposition~\ref{prop:normal} below) that $\polytope(\quiver,\theta)$ is a lattice polyhedron in the sense of Definition 7.1.3 in \cite{cox-little-schenck}. Denote by $X_{\polytope(\quiver,\theta)}$ the toric variety belonging to the normal fan of $\polytope(\quiver,\theta)$, see for example Theorem 7.1.6 in \cite{cox-little-schenck}. 

\begin{proposition}\label{prop:hille} 
We have 
\[\moduli(\quiver,\theta)\cong X_{\polytope(\quiver,\theta)}.\]
\end{proposition} 

\begin{proof} For quivers with no oriented cycles this is explained in \cite{altmann-hille} using a description of the fan of $\moduli(\quiver,\theta)$ in \cite{hille:canada}. An alternative explanation is the following: 
the lattice points in $\polytope(\quiver,n\theta)$ correspond bijectively to a $\mc$-basis in $\coord(\rep(\quiver))_{n\theta}$, namely assign to 
$m\in\polytope(Q,n\theta)\cap\mz^{\quiver_1}$ the function $R\mapsto R^m:=\prod_{a\in\quiver_1} R(a)^{m(a)}$. 
Now $X_{\polytope(\quiver,\theta)}$ is the projective spectrum of $\bigoplus_{n=0}^\infty\coord(\rep(\quiver))_{n\theta}$ (see Proposition 7.1.13 in \cite{cox-little-schenck}), 
just like $\moduli(\quiver,\theta)$.  
\end{proof} 

A more explicit description of $\moduli(\quiver,\theta)$ is possible thanks to normality of quiver polyhedra: 

\begin{proposition}\label{prop:normal} 
(i) Denote by $\quiver^1,\dots,\quiver^t$ the maximal subquivers of $\quiver$ that contain no oriented cycles. Then $\polytope(\quiver,\theta)\cap\mz^{\quiver_1}$ has a Minkowski sum decomposition 
\begin{equation}\label{eq:recessioncone} 
\polytope(\quiver,\theta)\cap\mz^{\quiver_1}=\polytope(\quiver,0)\cap\mz^{\quiver_1}+\bigcup_{i=1}^t\polytope(\quiver^i,\theta)\cap\mz^{\quiver_1}.\end{equation}

(ii) The quiver polyhedron $\polytope(\quiver,\theta)$ is a normal lattice polyhedron. 
\end{proposition}

\begin{proof} (i)  By the {\it support} of  $x\in \mr^{\quiver_1}$ we mean the set $\{a\in\quiver_1\mid x(a)\neq 0\}\subseteq \quiver_1$. 
It is obvious that $\polytope(\quiver,\theta)\cap\mz^{\quiver_1}$ contains the set on the right hand side of \eqref{eq:recessioncone}. 
To show the reverse inclusion take an $x\in\polytope(\quiver,\theta)\cap \mz^{\quiver_1}$. If its support contains no oriented cycles, then $x\in\polytope(\quiver^i,\theta)$ for some $i$. 
Otherwise take a minimal oriented cycle $C\subseteq\quiver_1$ in the support of $x$. Denote by $\varepsilon_C\in\mr^{\quiver_1}$ the characteristic function of $C$, and denote by $\lambda$ the minimal coordinate of $x$ along the cycle $C$. Then $\lambda\varepsilon_C\in\polytope(\quiver,0)$ and 
$y:=x-\lambda\varepsilon_C\in \polytope(\quiver,\theta)$. Moreover, $y$ has strictly smaller support than $x$. By induction on the size of the support we are done. 

(ii) The same argument as in (i) yields 
$\polytope(\quiver,\theta)=\polytope(\quiver,0)+\bigcup_{i=1}^t\polytope(\quiver^i,\theta)$. So $\polytope(\quiver,0)$ is the {\it recession cone}  of $\polytope(\quiver,\theta)$, 
and the set of vertices of $\polytope(\quiver,\theta)$ is contained in the union of the vertex sets of $\polytope(\quiver^i,\theta)$. As we pointed out before, the vertices of 
$\polytope(\quiver^i,\theta)$ belong to $\mz^{\quiver_1}$ by Theorem 13.11 in \cite{schrijver}, whereas the cone $\polytope(\quiver,0)$ is obviously rational and strongly convex. This shows that $\polytope(\quiver,\theta)$ is a lattice polyhedron in the sense of Definition 7.1.3 in \cite{cox-little-schenck}. 
For normality we need to show that for all positive integers $k$ we have $\polytope(\quiver,k\theta)\cap \mz^{\quiver_1}=k(\polytope(\quiver,\theta)\cap\mz^{\quiver_1})$ (the Minkowski sum of $k$ copies of $\polytope(\quiver,\theta)\cap\mz^{\quiver_1}$), see Definition 7.1.8 in \cite{cox-little-schenck}. Flow polytopes are normal by Theorem  13.14 in \cite{schrijver}, hence the $\polytope(\quiver^i,\theta)$ are normal for  $i=1,\dots,t$.  
So by (i)  we have $\polytope(\quiver,k\theta)\cap \mz^{\quiver_1}=\polytope(\quiver,0)\cap\mz^{\quiver_1}+\bigcup_{i=1}^t(\polytope(\quiver^i,k\theta)\cap \mz^{\quiver_1})
=\polytope(\quiver,0)\cap \mz^{\quiver_1}+\bigcup_{i=1}^tk(\polytope(\quiver^i,\theta)\cap \mz^{\quiver_1} )
\subseteq k(\polytope(\quiver,0)+\bigcup_{i=1}^t\polytope(\quiver^i,\theta)\cap\mz^{\quiver1})$.  
\end{proof} 

Let $C_1,\dots,C_r$ be the minimal oriented cycles (called also primitive cycles) in $\quiver$. Then their characteristic  functions $\varepsilon_{C_1},\dots,\varepsilon_{C_r}$ constitute a Hilbert basis in the monoid $\polytope(\quiver,0)\cap \mz^{\quiver_1}$. Enumerate the elements in 
$\{m,\varepsilon_{C_j}+m\mid m\in \bigcup_{i=1}^t\polytope(\quiver^i,\theta),j=1,\dots,r\}$ as 
$m_0,m_1,\dots,m_d$. For a lattice point $m\in \polytope(\quiver,\theta)\cap\mz^{\quiver_1}$ denote by $x^m:\rep(\quiver)\to \mc$ the function 
$x\mapsto \prod_{a\in\quiver_1} R(a)^{m(a)}$. Consider the map 
\begin{equation}\label{eq:rho}\rho:\rep(\quiver)^{\theta-ss}\to \mathbb{P}^d,\quad x\mapsto (x^{m_0}:\dots:x^{m_d}).\end{equation}

\begin{proposition}\label{prop:realizationofmoduli} 
$\moduli(\quiver,\theta)$ can be identified with the locally closed subset $\mathrm{Im}(\rho)$ in $\mathbb{P}^d$. 
\end{proposition} 

\begin{proof}  The morphism $\rho$ is $GL(1,\dots,1)$-invariant, hence it factors through the quotient morphism \eqref{eq:quotient}, so there exists a morphism 
$\mu:\moduli(\quiver,\theta)\to \mathrm{Im}(\rho)$ with $\mu\circ\pi=\rho$. One can deduce from Proposition~\ref{prop:normal} by the Proj construction of $\moduli(\quiver,\theta)$ that $\mu$ is an isomorphism. 
\end{proof} 

This shows also that there is a projective morphism $\moduli(\quiver,\theta)\to \moduli(\quiver,0)$. In particular, $\moduli(\quiver,\theta)$ is a projective variety if and only if $\quiver$ has no oriented cycles, i.e. if $\polytope(\quiver,\theta)$ is a polytope.

%%%%%%%%%%%%%%%%%%%%%%%%%%%%%%%%%%%%%

\section{Contractable arrows}\label{sec:reductions}

Throughout this section $\quiver$ stands for a quiver  and $\theta\in\mz^{\quiver_0}$ for a weight such that $\polytope(\quiver,\theta)$ is non-empty. 
For an undirected graph $\Gamma$ we set
$\chi(\Gamma):=|\Gamma_1|-|\Gamma_0|+\chi_0(\Gamma)$, where $\Gamma_0$ is the set of vertices, $\Gamma_1$ is the set of edges in $\Gamma$, and $\chi_0(\Gamma)$ is the number of connected components of $\Gamma$.   Define $\chi(\quiver):=\chi(\Gamma)$ and $\chi_0(\quiver):=\chi_0(\Gamma)$ where $\Gamma$ is the underlying graph of $\quiver$, and we say that $\quiver$ is {\it connected} if $\Gamma$ is connected, i.e. if $\chi_0(\quiver)=1$. 
 Denote by $\floweqs:\mr^{\quiver_1}\to\mr^{\quiver_0}$ the map given by 
\begin{equation}\label{eq:flow} 
\floweqs(x)(v)=\sum_{a^+=v}x(a)-\sum_{a^-=v}x(a)\qquad (v\in\quiver_0). \end{equation}
By definition we have $\polytope(\quiver,\theta)=\floweqs^{-1}(\theta)\cap \mr_{\ge 0}^{\quiver_1}$. 
It is well known that the codimension in $\mr^{\quiver_0}$ of the image of $\floweqs$ equals $\chi_0(\quiver)$, hence 
$\dim_{\mr}(\floweqs^{-1}(\theta))=\chi(\quiver)$ for any $\theta\in \floweqs(\mr^{\quiver_1})$, implying 
that $\dim(\polytope(\quiver,\theta))\le \chi(\quiver)$, where by the {\it dimension of a polyhedron} we mean the dimension of its affine span. 

We say that we {\it contract an arrow} $a\in\quiver_1$ which is not a loop when we pass to the pair $(\hat\quiver,\hat\theta)$, where $\hat\quiver$ is obtained from $\quiver$ by 
removing $a$ and glueing its endpoints $a^-,a^+$ to a single vertex $v\in\hat\quiver_0$, and setting $\hat\theta(v):=\theta(a^-)+\theta(a^+)$ whereas $\hat\theta(w)=\theta(w)$ for all vertices $w\in\hat\quiver_0\setminus \{v\}=\quiver_0\setminus\{a^-,a^+\}$. 

\begin{definition}\label{def:contractable} {\rm  Let $\quiver$ be a quiver, $\theta \in \mz^{\quiver_0}$ a weight such that $\polytope(\quiver,\theta)$ is non-empty. 
\begin{itemize}
\item[(i)]  An arrow $a\in\quiver_1$ is said to be {\it removable} if $\polytope(\quiver,\theta)$ is integral-affinely equivalent to $\polytope(\quiver',\theta)$, where $\quiver'$ is obtained from $\quiver$ by removing the arrow $a$: $\quiver'_0=\quiver_0$ and $\quiver'_1=\quiver_1\setminus \{a\}$.  
\item[(ii)] An arrow $a\in\quiver_1$ is said to be {\it contractable} if 
$\polytope(\quiver,\theta)$ is integral-affinely equivalent to $\polytope(\hat\quiver,\hat\theta)$, where $(\hat\quiver,\hat\theta)$ is obtained from $(\quiver,\theta)$ by 
contracting the arrow $a$. 
\item[(iii)] The pair $(\quiver,\theta)$ is called {\it tight} if there is no removable or contractable arrow in $\quiver_1$.  
\end{itemize} }
\end{definition}  

An immediate corollary of Definition~\ref{def:contractable} is the following statement: 

\begin{proposition}\label{prop:tightsufficient} 
Any quiver polyhedron $\polytope(\quiver,\theta)$ is integral-affinely equivalent to some $\polytope(\quiver',\theta')$, where $(\quiver',\theta')$ is tight. Moreover, $(\quiver',\theta')$  is obtained from $(\quiver,\theta)$ by successively removing or contracting arrows. 
\end{proposition} 

\begin{remark}\label{remark:altmann-straten-tight} 
{\rm A pair $(\quiver,\theta)$ is tight if and only if all its connected components are $\theta$-tight in the sense of Definition 12 of \cite{altmann-straten}; 
this follows from Lemma 7, Corollary 8, and Lemma 13 in \cite{altmann-straten}. These results imply also Corollary~\ref{cor:tightfacets} below, for which we give a direct derivation from Definition~\ref{def:contractable}.}
\end{remark}

\begin{lemma}\label{lemma:contractable} 
\begin{itemize} 
\item[(i)] The arrow $a$ is removable if and only if $x(a)=0$ for all $x\in\polytope(\quiver,\theta)$.  
\item[(ii)] The arrow $a$ is contractable if and only if 
in the affine space $\floweqs^{-1}(\theta)$ the halfspace $\{x\in \floweqs^{-1}(\theta)\mid x(a)\geq 0\}$ contains the polyhedron 
$\{x\in\floweqs^{-1}(\theta)\mid x(b)\geq 0 \quad \forall b\in\quiver_1\setminus\{a\}\}$.  
\end{itemize}
\end{lemma} 

\begin{proof} (i) is trivial. To prove (ii) denote by $\hat\quiver,\hat\theta$ the quiver and weight obtained by contracting $a$. Since the set of arrows of $\hat\quiver$ can be identified with $\hat\quiver_1=\quiver_1\setminus\{a\}$, we have the projection map $\pi:\floweqs^{-1}(\theta)\to \floweqs'^{-1}(\hat\theta)$ obtained by forgetting the coordinate $x(a)$. The equation 
\[x(a)=\theta(a^+)-\sum_{b\in\quiver_1\setminus\{a\},b^+=a^+}x(b)+\sum_{b\in\quiver_1\setminus\{a\},b^-=a^+}x(b)\] 
shows that $\pi$ is injective, hence it gives an affine linear isomorphism $\floweqs^{-1}(\theta)\cap\mz^{\quiver_1}$ and $\floweqs'^{-1}(\hat\theta)\cap\mz^{\hat\quiver_1}$, and maps injectively the lattice polyhedron $\polytope(\quiver,\theta)$ onto an integral-affinely equivalent lattice polyhedron contained in $\polytope(\hat\quiver,\hat\theta)$.  Thus $a$ is contractable if and only if on the affine space $\floweqs^{-1}(\theta)$ the inequality $x(a)\ge 0$ is a consequence of the inequalities 
$x(b)\ge 0$ $(b\in\quiver_1\setminus \{a\}$). 
\end{proof} 

For an arrow $a\in\quiver_1$ set $\polytope(\quiver,\theta)_{x(a)=0}:=\{x\in\polytope(\quiver,\theta)\mid x(a)=0\}$. 

\begin{corollary}\label{cor:tightfacets} (i) The pair $(\quiver,\theta)$ is tight if and only if the assignment 
$a\mapsto \polytope(\quiver,\theta)_{x(a)=0}$ gives a bijection between $\quiver_1$ and the facets (codimension $1$ faces) of $\polytope(\quiver,\theta)$. 

(ii) If $(\quiver,\theta)$ is tight, then $\dim(\polytope(\quiver,\theta))=\chi(\quiver)$. 
\end{corollary} 

\begin{proof} Lemma~\ref{lemma:contractable} shows that   $(\quiver,\theta)$ is tight  if and only if $\affspan(\polytope(\quiver,\theta))=\floweqs^{-1}(\theta)$ and $\{x(a)=0\}\cap\floweqs^{-1}(\theta)$ ($a\in\quiver_1$) are distinct supporting hyperplanes of $\polytope(\quiver,\theta)$ in its affine span.  
\end{proof} 

The following simple sufficient condition for contractibility of an arrow turns out to be sufficient for our purposes.  
For a subset $S\subseteq\quiver_0$ set $\theta(S):=\sum_{v\in S}\theta(v)$. By \eqref{eq:flow} for $x\in\floweqs^{-1}(\theta)$ we have 
\begin{equation}\label{eq:theta(S)}\theta(S)=\sum_{a\in \quiver_1, a^+\in S}x(a)-\sum_{a\in\quiver_1,a^-\in S}x(a)=\sum_{a^+\in S,a^-\notin S} x(a)-
\sum_{a^-\in S,a^+\notin S} x(a).
\end{equation} 

\begin{proposition}\label{prop:oneinoneout}  
Suppose that $S\subset \quiver_0$ has the property that there is at most one arrow $a$ with $a^+\in S$, $a^-\notin S$ and at most one arrow $b$ with $b^+\notin S$ and 
$b^-\in S$.  Then $a$ (if exists) is contractable when $\theta(S)\ge 0$ and $b$ (if exists) is contractable when $\theta(S)\le 0$. 
\end{proposition} 

\begin{proof} By \eqref{eq:theta(S)} we have $\theta(S)=x(a)-x(b)$, hence by Lemma~\ref{lemma:contractable} $a$ or $b$ is contractable, depending on the sign of $\theta(S)$.  
\end{proof} 

By the {\it valency} of a vertex $v\in\quiver_0$ we mean $|\{a\in\quiver_1\mid a^-=v\}|+|\{a\in\quiver_1\mid a^+=v\}|$. 

\begin{corollary}\label{cor:shrinking} 
(i) Suppose that the vertex $v\in \quiver_0$  has valency $2$, and $a,b\in\quiver_1$ are arrows such that $a^+=b^-=v$. 
Then the arrow $a$ is contractable when  $\theta(v)\ge 0$ and $b$ is contractable when $\theta(v)\le 0$. 

(ii) Suppose that for some $c\in\quiver_1$, $c^-$ and $c^+$ have valency $2$, and $a,b\in\quiver _1\setminus \{c\}$ with $a^-=c^-$ and $b^+=c^+$. 
Then $a$ is contractable when $\theta(c^-)+\theta(c^+)\le 0$ and $b$ is contractable when $\theta(c^-)+\theta(c^+)\ge 0$. 
\end{corollary} 
\begin{proof} Apply Proposition~\ref{prop:oneinoneout} with $S=\{v\}$ to get (i) and with $S=\{c^-,c^+\}$ to get (ii). 
\end{proof}

\begin{proposition}\label{prop:reflection} 
Suppose that there are 
exactly two arrows $a,b\in\quiver_1$ attached to some vertex $v$, and either $a^+=b^+=v$ or $a^-=b^-=v$.  
Let $\quiver',\theta'$ be the quiver and weight obtained after replacing 

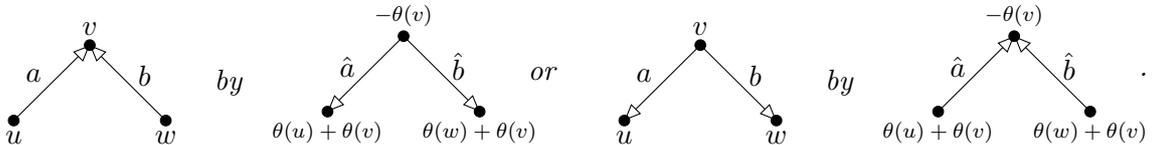
\begin{tikzpicture}[>=open triangle 45] 
\node [right] at (2.5,0.5) {by}; 
\node [below] at (0,0) {$u$}; \node [below] at (2,0) {$w$}; 
\foreach \x in {(0,0),(1,1),(2,0)} \filldraw \x circle (2pt); 
\node [above] at (1,1) {$v$};  \node [above, left] at (0.5,0.6) {$a$}; \node[above,right] at (1.5,0.6) {$b$}; 
\draw [->] (0,0)--(1,1); \draw [<-] (1,1)--(2,0);  
\end{tikzpicture} 
\begin{tikzpicture}[>=open triangle 45] 
\node [right] at (2.5,0.5) {or}; 
\node [below] at (0,0) {\scriptsize{$\theta(u)+\theta(v)$}}; \node [below] at (2,0) {\scriptsize{$\theta(w)+\theta(v)$}}; 
\foreach \x in {(0,0),(1,1),(2,0)} \filldraw \x circle (2pt); 
\node [above] at (1,1) {\scriptsize{$ -\theta(v)$}};  \node [above, left] at (0.5,0.6) {$\hat a$}; \node[above,right] at (1.5,0.6) {$\hat b$}; 
\draw [<-] (0,0)--(1,1); \draw [->] (1,1)--(2,0);  
\end{tikzpicture} \quad 
\begin{tikzpicture}[>=open triangle 45] 
\node [right] at (2.5,0.5) {by}; 
\node [below] at (0,0) {$u$}; \node [below] at (2,0) {$w$}; 
\foreach \x in {(0,0),(1,1),(2,0)} \filldraw \x circle (2pt); 
\node [above] at (1,1) {$v$};  \node [above, left] at (0.5,0.6) {$a$}; \node[above,right] at (1.5,0.6) {$b$}; 
\draw [<-] (0,0)--(1,1); \draw [->] (1,1)--(2,0);  
\end{tikzpicture} 
\begin{tikzpicture}[>=open triangle 45] 
\node [right] at (2.5,0.5) {.}; 
\node [below] at (0,0) {\scriptsize{$\theta(u)+\theta(v)$}}; \node [below] at (2,0) {\scriptsize{$\theta(w)+\theta(v)$}}; 
\foreach \x in {(0,0),(1,1),(2,0)} \filldraw \x circle (2pt); 
\node [above] at (1,1) {\scriptsize{$ -\theta(v)$}};  \node [above, left] at (0.5,0.6) {$\hat a$}; \node[above,right] at (1.5,0.6) {$\hat b$}; 
\draw [->] (0,0)--(1,1); \draw [<-] (1,1)--(2,0);  
\end{tikzpicture} \quad 

\noindent That is, replace the arrows $a,b$ by  $\hat a$ and $\hat b$ obtained by reversing them, and 
consider the weight $\theta'\in\mz^{\quiver'_1}$ given by $\theta'(v)=-\theta(v)$, $\theta'(u)=\theta(u)+\theta(v)$ when $u\neq v$ is an endpoint of $a$ or $b$, and 
$\theta'(w)=\theta(w)$ for all other $w\in\quiver'_0=\quiver_0$.  
Then the polyhedra $\polytope(\quiver,\theta)$ and $\polytope(\quiver',\theta')$ are integral-affinely equivalent. 
\end{proposition}

\begin{proof} It is straightforward to check that the map $\varphi:\mr^{\quiver_1}\to \mr^{\quiver'_1}$ given by 
$\varphi(x)(\hat a)=x(b)$, $\varphi(x)(\hat b)=x(a)$, and $\varphi(x)(c)=x(c)$ for all $c\in \quiver'_1\setminus\{\hat a,\hat b\}=\quiver_1\setminus\{a,b\}$ 
restricts to an isomorphism 
between $\affspan(\polytope(\quiver,\theta))$ and $\affspan(\polytope(\quiver',\theta'))$ 
satisfying (i) and (ii) in Definition~\ref{def:isomorphicpolytopes}. 
\end{proof}

\begin{remark}\label{remark:reflection} {\rm 
Proposition~\ref{prop:reflection} can be interpreted in terms of {\it reflection transformations}: it was shown in Sections 2 and 3 in \cite{kac} (see also Theorem 23 in \cite{skowronski-weyman}) that   reflection transformations on representations of quivers induce isomorphisms of algebras of semi-invariants. 
Now under our assumptions a reflection transformation at vertex $v$ fixes the dimension vector $(1,\dots,1)$. }
\end{remark}

\begin{proposition}\label{prop:glueing} 
Suppose that $\quiver$ is the union of its full subquivers $\quiver'$, $\quiver''$ which are either disjoint or have a single common vertex $v$. Identify $\mr^{\quiver'_1}\oplus \mr^{\quiver''_1}=\mr^{\quiver_1}$ in the obvious way, 
and let $\theta'\in\mz^{\quiver'_0}\subset\mz^{\quiver_0}$, $\theta''\in\mz^{\quiver''_0}\subset\mz^{\quiver_0}$ be the unique weights with 
$\theta=\theta'+\theta''$ and  $\theta'(v)=-\sum_{w\in \quiver'_0\setminus \{v\}}\theta(w)$, $\theta''(v)=-\sum_{w\in \quiver''_0\setminus\{v\}}\theta(w)$ when 
$\quiver'_0\cap\quiver''_0=\{v\}$.

(i) Then the quiver polyhedron $\polytope(\quiver,\theta)$ is the product of the polyhedra $\polytope(\quiver',\theta')$ 
and $\polytope(\quiver'',\theta'')$.  

(ii) We have $\moduli(\quiver,\theta)\cong \moduli(\quiver',\theta')\times \moduli(\quiver'',\theta'')$.  
\end{proposition} 

\begin{proof} (i) A point $x\in\mr^{\quiver_1}$ uniquely decomposes as $x=x'+x''$, where $x'(a)=0$ for all $a\notin\quiver'_1$ and 
$x''(a)=0$ for all $a\notin \quiver''_1$. It is obvious by definition of quiver polyhedra that $x\in\polytope(\quiver,\theta)$ if and only if 
$x'\in\polytope(\quiver',\theta')$ and $x''\in\polytope(\quiver'',\theta'')$. 

(ii) was observed already in  \cite{hille:chemnitz} and follows from (i) by Proposition~\ref{prop:hille}.  
\end{proof}

\begin{definition} \label{def:prime} {\rm  
\begin{itemize} 
\item[(i)]  We  call a connected undirected graph $\Gamma$   (with at least one edge) {\it prime} if it is not the union of full proper subgraphs  $\Gamma',\Gamma''$ having only one common vertex (i.e. it is 2-vertex-connected). A quiver $\quiver$ will be called  {\it prime} if  its underlying graph is prime. 

\item[(ii)] We call  a toric variety {\it prime} if it is not the product of lower dimensional toric varieties. \end{itemize}}
\end{definition} 

Obviously any toric variety is the product of prime toric varieties, and this product decomposition is unique up to the order of the factors  (see for example Theorem 2.2 in \cite {hatanaka}). It is not immediate from the definition, but we shall show in Theorem~\ref{thm:directprod} (iii) that the prime factors of a toric quiver variety are  quiver 
varieties as well.   

Note that a toric quiver variety associated to a non-prime quiver may well be prime, and conversely, a toric quiver variety associated to a prime quiver can be non-prime, as it is shown by the following example: 

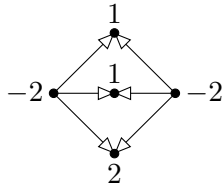
\[\begin{tikzpicture}[>=open triangle 45,scale=0.8] 
\foreach \x in {(0,0),(1,0),(1,1),(1,-1),(2,0)} \filldraw \x circle (2pt); 
\foreach \y in {(1,0),(1,1),(1,-1)} \draw  [->]  (0,0)--\y; 
\foreach \x in {(1,0),(1,1),(1,-1)} \draw  [<-] \x--(2,0); 
\node [left] at (0,0) {$-2$};  \node [above] at (1,1) {$1$}; \node [above] at (1,0) {$1$}; \node [below] at (1,-1) {$2$}; \node [right] at (2,0) {$-2$}; 
\end{tikzpicture}\]

The quiver in the picture is prime but the moduli space corresponding to this weight is $\mathbb{P}^1\times\mathbb{P}^1$. 
However, as shown by Theorem~\ref{thm:directprod} below, when the tightness of some $(\quiver, \theta)$ is assumed, decomposing $\quiver$ into its unique maximal prime components gives us the unique decomposition of $\moduli(\quiver,\theta)$ as a product of prime toric varieties. 

\begin{theorem} \label{thm:directprod}
\begin{itemize}
\item[(i)] Let $\quiver^i$ $(i=1,\dots,k)$ be the maximal prime full subquivers of $\quiver$, and denote by  $\theta^i\in\mz^{\quiver^i_0}$ the unique weights satisfying $\sum_{i=1}^k\theta^i(v)=\theta(v)$ for all $v\in\quiver_0$ and $\sum_{v\in{\quiver_0^i}}\theta^i(v)=0$ for all $i$. Then $\moduli(\quiver,\theta)\cong \prod_{i=1}^k\moduli(\quiver^i,\theta^i)$. Moreover, if $(\quiver,\theta)$ is tight, then  the $(\quiver^i,\theta^i)$ are all tight. 
\item[(ii)] If $(\quiver,\theta)$ is tight then $\moduli(\quiver,\theta)$ is prime if and only if $\quiver$ is prime. 
\item[(iii)]  Any toric quiver variety is the product of prime toric quiver varieties. 
\end{itemize}
\end{theorem}

\begin{proof}
The isomorphism $\moduli(\quiver,\theta)\cong \prod_{i=1}^k\moduli(\quiver^i,\theta^i)$ follows from Proposition~\ref{prop:glueing}. The second statement in (i) follows from this isomorphism  and Corollary~\ref{cor:tightfacets}. 

Next we turn to the proof of (ii), so suppose that $(\quiver,\theta)$ is tight. If $\quiver$ is not prime, then $\chi(\quiver^i)>0$ for all $i$, hence $\moduli(\quiver,\theta)$ is not prime by (i). To show the reverse implication assume on the contrary that $\quiver$ is prime, and $\moduli(\quiver,\theta)\cong X'\times X''$ where 
$X',X''$ are positive dimensional toric varieties.  Note that then $\quiver_1$ does not contain loops. 
Let $\{\varepsilon_a\mid a \in \quiver_1\}$ be a $\mz$-basis of $\mz^{\quiver_1}$, and for each vertex $v\in\quiver_0$ let us define $C_v := \sum_{a^+=v}\varepsilon_a - \sum_{a^-=v}\varepsilon_a$. Following the description of the toric fan $\Sigma$ of  $\moduli(\quiver,\theta)$ in \cite{hille:canada} we can identify the lattice of one-parameter subgroups $N$ of $\moduli(\quiver,\theta)$ with $\mz^{\quiver_1} / \langle C_v\mid v\in\quiver_0 \rangle$, and the ray generators of the fan with the cosets of the 
$\varepsilon_a$. Denoting by $\Sigma'$ and $\Sigma''$ the fans of $X'$ and $X''$ respectively, we have 
$\Sigma = \Sigma' \times \Sigma''=\{\sigma'\times\sigma''\mid\sigma'\in\Sigma',\:\sigma''\in\Sigma''\}$ (see \cite{cox-little-schenck} Proposition 3.1.14). 
Denote by $\pi':N\rightarrow N'$, $\pi'':N\rightarrow N''$ the natural projections to the sets of  one-parameter subgroups of the tori in $X'$ and $X''$. 
For each ray generator $\varepsilon_a$ we have either $\pi'(\varepsilon_a) = 0$ or $\pi''(\varepsilon_a) = 0$. Since $(\quiver,\theta)$ is tight we obtain a partition of 
$\quiver_1$ into two disjoint non-empty sets of arrows: $\quiver_1'=\{a\in\quiver_1\mid\pi''(a)=0\}$ and $\quiver_1''=\{a\in\quiver_1\mid\pi'(a)=0\}$.  
Since $\quiver$ is prime, it is connected, hence there exists a vertex $w$ incident to arrows both from $\quiver'_1$ and $\quiver''_1$. 
Let $\Pi'$ and $\Pi''$ denote the projections from $\mz^{\quiver_1}$ to $\mz^{\quiver_1'}$ and $\mz^{\quiver_1''}$. 
By choice of $w$ we have $\Pi'(C_w)\neq 0$  and $\Pi''(C_w)\neq 0$.   Writing $\varphi$ for the natural map from 
$\mz^{\quiver_1}$ to $N \cong \mz^{\quiver_1} / \langle C_v\mid v\in\quiver_0 \rangle$ 
we have $\varphi \circ \Pi' = \pi' \circ \varphi$ and $\varphi \circ \Pi'' = \pi'' \circ \varphi$, so $\ker(\varphi) = \langle C_v\mid v\in\quiver_0 \rangle$ is closed under 
$\Pi'$ and $\Pi''$. Taking into account that $\sum_{v\in\quiver_0}C_v=0$ we deduce that  
$\Pi'(C_w)=\sum_{v\in\quiver_0\setminus\{w\}}\lambda_vC_v$ for some $\lambda_v\in \mz$. 
Set $S':=\{v\in\quiver_0\mid\lambda_v\ne 0\}$. Since each arrow appears in exactly two of the $C_v$, it follows that $S'$ contains all vertices connected to $w$ by an arrow in $\quiver'_1$, hence $S'$ is non-empty. Moreover, the set of arrows having exactly one endpoint in $S'$ are exactly those arrows in $\quiver'_1$ that are adjacent to $w$. 
Thus $S'':=\quiver_0\setminus (S'\cup\{w\})$ contains all vertices that are connected to $w$ by an arrow from $\quiver''_1$, hence $S''$ is non-empty. Furthermore, there are no arrows in $\quiver_1$ connecting a vertex from $S'$ to a vertex in $S''$. It follows that $\quiver$ is the union of its full subquivers spanned by the vertex sets 
$S'\cup\{w\}$ and $S''\cup\{w\}$, having only one common vertex $w$ and no common arrow. This contradicts the assumption that $\quiver$ was prime. 

Statement (iii) follows from (i), (ii) and Proposition~\ref{prop:tightsufficient}. 
\end{proof}

Note that if $\chi(\Gamma)\ge 2$ and $\Gamma$ is prime, then $\Gamma$ contains no loops (i.e. an edge with identical endpoints), every vertex of $\Gamma$ has valency at least $2$, and $\Gamma$ has  at least two vertices with valency at least $3$. 

\begin{definition}\label{def:skeleton} {\rm 
For $d=2,3,\dots$ denote by $\sklist_d$ the set of prime graphs $\Gamma$ with $\chi(\Gamma)=d$ in which all vertices have valency at least $3$. 
Let $\reducedquivers_d$ stand for the set of quivers $\quiver$ with no oriented cycles obtained from a graph $\Gamma\in\sklist_d$ by orienting some of the edges somehow and putting a sink on the remaining edges (that is, we replace an edge  by a path of length $2$ in which both edges are pointing towards the new vertex in the middle). We shall call $\Gamma$ the {\it skeleton} $\skeleton(\quiver)$ of $\quiver$; note that $\chi(\quiver)=\chi(\skeleton(\quiver))$. 
}\end{definition}

Starting from $\quiver$, its skeleton $\Gamma=\skeleton(\quiver)$ can be recovered as follows: 
$\Gamma_0$ is the subset of $\quiver_0$ consisting of the valency $3$ vertices. 
For each path in the underlying graph of $\quiver$ that connects two vertices in $\Gamma_0$ and whose inner vertices have valency $2$ we put an edge. 
Clearly, a  quiver $\quiver$ with $\chi(\quiver)=d\ge 2$ belongs to $\reducedquivers_d$ if and only if the following conditions hold: 
\begin{itemize}
\item[(i)] $\quiver$ is prime.  
\item[(ii)] There is no arrow of $\quiver$ connecting valency $2$ vertices. 
\item[(iii)] Every valency $2$ vertex of $\quiver$ is a sink. 
\end{itemize} 
Furthermore, set $\reducedquivers:=\bigsqcup_{d=1}^\infty \reducedquivers_d$ where $\reducedquivers_1$ is the $1$-element set consisting of the Kronecker quiver 
\begin{tikzpicture}[>=open triangle 45]  \draw [->] (0,0) to [out=45, in=135] (1,0); \node [above] at (0,0) {}; \node [above] at (1,0) {}; 
\draw  [->] (0,0) to [out=315, in=225] (1,0) ;
\filldraw (0,0) circle (1.5pt) (1,0) circle (1.5pt); 
\end{tikzpicture}.

\begin{remark}\label{remark:combinatorial-tightness}  {\rm 
A purely  combinatorial characterization of tightness is given in Lemma 13 of  \cite{altmann-straten}. Namely, $(\quiver,\theta)$ is tight if and only if  any connected 
component of $\quiver$ is $\theta$-stable, and  any connected component  of $\quiver\setminus \{a\}$ for any $a\in\quiver_1$ is $\theta$-stable (see Section~\ref{sec:affine} for the notion of $\theta$-stability). In the same Lemma  it is also shown that if $(\quiver,\theta)$ is tight, then $(\quiver,\delta_{\quiver})$ is tight, where $\delta_{\quiver}:=\sum_{a\in\quiver_1}(\varepsilon_{a^+}-\varepsilon_{a^-})$ is the so-called {\it canonical weight} (here $\varepsilon_v$ stands for the characteristic function of $v\in\quiver_0$). 
It is easy to deduce that for a connected quiver $\quiver$ the pair $(\quiver,\delta_{\quiver})$ is tight if and only if there is no partition $\quiver_0=S\coprod S'$ such that there is 
at most one arrow from $S$ to $S'$ and there is at most one arrow from $S'$ to $S$. }
\end{remark}

\begin{proposition}\label{prop:boundonskeletons} 
For any $d\ge 2$, $\Gamma\in\sklist_d$ and $\quiver\in\reducedquivers_d$ we have the inequalities 
\[|\Gamma_0|\le 2d-2, \quad |\Gamma_1|\le 3d-3, \quad |\quiver_0|\le  5(d-1), \quad |\quiver_1|\le  6(d-1).\] 
In particular, $\sklist_d$ and $\reducedquivers_d$ are finite for each positive integer $d$. 
\end{proposition} 

\begin{proof} 
Take $\Gamma\in\sklist_d$ where $d\ge 2$. 
Then $\Gamma$ contains no loops, and denoting by $e$ the number of edges and by $v$ the number of vertices of $\Gamma$, we have the inequality 
$2e\ge 3v$, since each vertex is adjacent to at least three edges. On the other hand $e=v-1+d$. We conclude that $v\le 2d-2$ and hence $e\le 3d-3$. 
For $\quiver\in\reducedquivers_d$ with $\skeleton(\quiver)=\Gamma$ we have that $|\quiver_0|\le v+e$ and $|\quiver_1|\le 2e$. 
\end{proof}

\begin{theorem}\label{thm:finitelymanymoduli}  \begin{itemize}
\item[(i)] Any $d$-dimensional  prime toric quiver variety $\moduli(\quiver,\theta)$ can be realized by a tight pair $(\quiver,\theta)$ 
where  $\quiver\in\reducedquivers_d$ (consequently $|\quiver_0|\le 5(d-1)$ and $|\quiver_1|\le 6(d-1)$ when $d\ge 2$). 
\item[(ii)] For each positive integer $d$ up to isomorphism there are only finitely many $d$-dimensional toric quiver varieties.  
\end{itemize}
 \end{theorem} 

\begin{proof} It follows from Propositions~\ref{prop:tightsufficient}, Corollary~\ref{cor:shrinking} and Proposition~\ref{prop:reflection} that any $d$-dimensional  prime toric quiver variety  can be realized by a tight pair $(\quiver,\theta)$ where  $\quiver\in\reducedquivers_d$; the bounds on vertex and arrow sets of the quiver follow by 
Proposition~\ref{prop:boundonskeletons}. It remains to show (ii). 
For a given quiver $\quiver$ we say that the weights $\theta$ and $\theta'$ are equivalent if $\rep(\quiver)^{\theta-ss}=\rep(\quiver)^{\theta'-ss}$;  
this implies  that $\moduli(\quiver,\theta)=\moduli(\quiver,\theta')$. For a given representation $R$ of $\quiver$, the set of weights $\theta$ for which $R$ is $\theta$-semistable 
is determined by the set of dimension vectors of subrepresentations of $R$. Since there are finitely many  possibilities for the dimension vectors of a subrepresentation of a representation with dimension vector $(1,\dots,1)$, up to equivalence there  are only finitely many different weights, hence there are finitely many possible moduli spaces
for a fixed $\quiver$. \end{proof}

\begin{remark} {\rm
Part (i) of Theorem~\ref{thm:finitelymanymoduli} can be directly obtained from the results in \cite{altmann-nill-schwentner-wiercinska} and \cite{altmann-straten}. From the proof of Theorem 7 in \cite{altmann-nill-schwentner-wiercinska} it follows that the bound on the number of vertices and edges hold whenever the canonical weight is tight for a quiver. While in \cite{altmann-nill-schwentner-wiercinska} it is assumed that $\quiver$ has no oriented cycles, their argument for the bound applies to the general case as well. Moreover Lemma 13 in  \cite{altmann-straten} shows that every toric quiver variety can be realized by a pair $(\quiver,\theta)$ where $\quiver$ is tight with the canonical weight. These two results imply part (i) of Theorem~\ref{thm:finitelymanymoduli}.  
}
\end{remark}

\begin{remark}\label{remark:fourier-motzkin} {\rm 
We mention that for a fixed quiver $\quiver$ it is possible to give an algorithm to produce a representative for each of the finitely many equivalence classes of weights. 
The change of the moduli spaces of a given quiver when we vary the weight is studied in \cite{hille:chemnitz}, \cite{hille:canada}, where the inequalities determining the chamber system were given. To find an explicit weight in each chamber one can use the Fourier-Motzkin algorithm. }
\end{remark}

Theorem~\ref{thm:finitelymanymoduli} is sharp, and 
the reductions on the quiver are optimal, in the sense that in general one can not hope for reductions that would yield smaller quivers:   

\begin{proposition}\label{prop:optimal} 
For each natural number $d\ge 2$ there exists a $d$-dimensional toric quiver variety $\moduli(\quiver,\theta)$ with 
$|\quiver_1|=6(d-1)$, $|\quiver_0|=5(d-1)$,  such that  for any other quiver and weight $\quiver',\theta'$ with 
$\moduli(\quiver,\theta)\cong \moduli(\quiver',\theta')$ (isomorphism of toric varieties) we have that 
$|\quiver'_1|\ge |\quiver_1|$ and $|\quiver'_0|\ge |\quiver_0|$.  
\end{proposition} 

\begin{proof} The number of inequalities defining $\polytope(\quiver,\theta)$ in its affine span is obviously bounded by the number of arrows of $\quiver$, 
so the number of facets of a quiver polyhedron $\polytope(\quiver,\theta)$ is bounded by $|\quiver_1|$.  
On the other hand the number of facets is an invariant of the corresponding toric variety, as it equals the number of rays in the toric fan of $\moduli(\quiver,\theta)$. 
Therefore by Corollary~\ref{cor:tightfacets} it is sufficient to show the existence of a tight  $(\quiver,\theta)$ with $|\quiver_1|=6(d-1)$ and $|\quiver_0|=5(d-1)$. 
Such a pair $(\quiver,\theta)$ is provided in Example~\ref{example:2sharp}. 
\end{proof} 

\begin{example}\label{example:2sharp}
{\rm  For  $d\ge 3$ consider the graph below  with 
$2(d-1)$ vertices. Removing any two edges from this graph we obtain a connected graph. Now let $\quiver$ be the quiver obtained by putting a sink on each of the edges (so the graph below is the skeleton of $\quiver$).  Then $(\quiver,\delta_{\quiver})$ is tight by Remark~\ref{remark:combinatorial-tightness} ($\delta_{\quiver}$ takes value $2$ on each sink and value $-3$ on each source).  
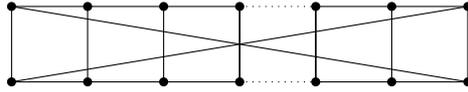
\[\begin{tikzpicture}
\draw (0,0)--(3,0)--(3,1)--(0,1)--(0,0) (4,0)--(6,0)--(6,1)--(4,1)--(4,0) (1,0)--(1,1) (2,0)--(2,1)  (3,0)--(3,1) (4,0)--(4,1) (5,0)--(5,1);
\draw [dotted] (3,0)--(4,0)  (3,1)--(4,1); 
\draw (0,1) to  (6,0); 
\draw (0,0) to (6,1);
\filldraw (0,0) circle (1.5pt)  (1,0) circle (1.5pt)  (2,0) circle (1.5pt)  (3,0) circle (1.5pt)  (4,0) circle (1.5pt)  (5,0) circle (1.5pt)  (6,0) circle (1.5pt) 
(0,1) circle (1.5pt)  (1,1) circle (1.5pt)  (2,1) circle (1.5pt)  (3,1) circle (1.5pt)  (4,1) circle (1.5pt)  (5,1) circle (1.5pt)  (6,1) circle (1.5pt);
\end{tikzpicture} 
\]}
\end{example} 

Relaxing the condition on tightness it is possible to come up with a shorter list of quivers whose moduli spaces exhaust all possible toric quiver varieties. 
A key role is played by the following statement: 

\begin{proposition}\label{prop:addanarrow} Suppose that $\quiver$ has no oriented cycles and $a\in\quiver_1$ is an arrow such that contracting it we get a quiver (i.e. the quiver $\hat\quiver$ described in Definition~\ref{def:contractable}) 
that has no oriented cycles. Then for a sufficiently   large integer $d$ we have that $a$ is contractable for the pair $(\quiver,\theta+d(\varepsilon_{a^+}-\varepsilon_{a^-}))$, 
where $\varepsilon_v\in\mz^{\quiver_0}$ stands for the characteristic function of $v\in\quiver_0$. 
\end{proposition} 

\begin{proof} Set $\psi_d = \theta+d(\varepsilon_{a^+}-\varepsilon_{a^-})$, and note that $\hat\psi_d = \hat\theta$ for all d. Considering the embeddings   $\pi:\floweqs^{-1}(\psi_d)\to \floweqs'^{-1}(\hat\theta)$ described in the proof of Lemma~\ref{lemma:contractable}, we have that for any $d$, any $y\in\floweqs^{-1}(\psi_d)$ and $b\in\quiver_1\setminus\{a\}$, \[\min\{x(b)\mid x\in\polytope(\hat\quiver,\hat\theta)\}\leq y(b) \leq \max\{x(b)\mid x\in\polytope(\hat\quiver,\hat\theta)\}\] 
Since we assumed that $\hat\quiver$ has no oriented cycles, the minimum and the maximum in the inequality above are finite. Now considering the arrows incident to $a^-$ we obtain that for any $x\in\floweqs^{-1}(\psi_d)$ we have $x(a) = d-\theta(a^-)+\sum_{b^+ = a^-}x(b) - \sum_{b^- = a^-,b\neq a}x(b)$. Thus for $d \geq \theta(a^-)-\min\{ \sum_{b^+ = a^-}x(b) - \sum_{b^- = a^-,b\neq a}x(b) \mid x\in\floweqs'^{-1}(\hat\theta)\}$ the arrow $a$ is contractable for $(\quiver,\psi_d)$ by Lemma~\ref{lemma:contractable}.
\end{proof} 

For $d\ge 2$ introduce a partial  ordering $\ge$ on $\sklist_d$: we set $\Gamma\ge \Gamma'$ if $\Gamma'$ is obtained from $\Gamma$ by contracting an edge, and take the transitive closure of this relation. 
Now for each positive integer $d\ge 2$ denote by $\sklist'_d\subseteq\sklist_d$ the set of undirected graphs $\Gamma\in \sklist_d$  that are maximal with respect to the relation $\ge$, and set $\sklist'_1:=\sklist_1$. 
It is easy to see that for $d\ge 2$,  $\sklist'_d$ consists of $3$-regular graphs (i.e. graphs in which all vertices have valency $3$). Now denote by $\reducedquivers'_d$ the quivers which are obtained by putting a sink on each edge from a graph from $\sklist'_d$. 

\begin{theorem}\label{thm:3regular} 
For $d\ge 2$ any prime $d$-dimensional projective toric quiver variety is isomorphic to $\moduli(\quiver,\theta)$ where $\quiver\in\reducedquivers'_d$.
\end{theorem}

\begin{proof} This is an immediate consequence of Theorem~\ref{thm:finitelymanymoduli} and Proposition~\ref{prop:addanarrow}. 
\end{proof} 

\begin{example}\label{example:threedim} {\rm $\sklist'_3$ consists of two graphs: 

\begin{tikzpicture} 
\foreach \x in {(0,0),(1,0),(0,1),(1,1)} \filldraw \x circle (2pt); 
\draw (0,0)--(1,0)--(1,1)--(0,1)--(0,0); \draw (0,0) to [out=135,in=225] (0,1); \draw (1,0) to [out=45,in=315] (1,1);
\end{tikzpicture} 
\qquad \qquad 
\begin{tikzpicture} 
\foreach \x in {(0,0),(1,0),(0,1),(1,1)} \filldraw \x circle (2pt); 
\draw (0,0)--(1,0)--(1,1)--(0,1)--(0,0); \draw (0,0)--(1,1); \draw (1,0)--(0,1);
\end{tikzpicture} 

Now put a sink on each edge of the above graphs. The first of the two resulting quivers is not tight for the canonical weight. After tightening we obtain the following two quivers among whose moduli spaces all $3$-dimensional prime projective toric quiver varieties occur: 

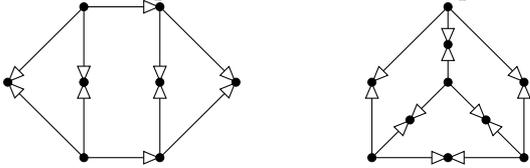
\begin{tikzpicture}[>=open triangle 45]   
\foreach \x in {(0,0),(0,1),(0,2),(-1,1),(0,1),(1,0),(1,2),(1,1),(2,1)} \filldraw \x circle (1.5pt); 
 \draw [->] (0,0) to (1,0);  \draw [->] (0,2) to (1,2);  \draw [->] (0,0) to (0,1);  \draw [->] (0,2) to (0,1);   \draw [->] (0,0) to (-1,1);   \draw [->] (0,2) to (-1,1);   \draw [->] (1,2) to (1,1); \draw [->] (1,0) to (1,1);  \draw [->] (1,0) to (2,1);  \draw [->] (1,2) to (2,1);
\end{tikzpicture} 
\qquad \qquad 
\begin{tikzpicture}[>=open triangle 45]   
\foreach \x in {(0,0),(1,0),(2,0),(1,1),(1,2),(0,1),(2,1),(0.5,0.5),(1.5,0.5),(1,1.5)} \filldraw \x circle (1.5pt); 
\draw [->] (0,0) to (0,1); \draw [->] (1,2) to (0,1); \draw [->] (0,0) to (1,0); \draw [->] (2,0) to (1,0); \draw [->] (2,0) to (2,1); \draw [->] (1,2) to (2,1); \draw [->] (0,0) to (0.5,0.5); \draw [->] (1,1) to (0.5,0.5); \draw [->] (1,1) to (1,1.5); \draw [->] (1,2) to (1,1.5);  \draw [->] (2,0) to (1.5,0.5); \draw [->] (1,1) to (1.5,0.5);
\end{tikzpicture}
 
}\end{example}

%%%%%%%%%%%%%%%%%%%%%%%%%%%%%%%%%%%

\section{The $2$-dimensional case} \label{sec:2-dimensional} 

As an illustration of the general classification scheme explained in Section~\ref{sec:reductions}, 
we quickly reproduce the classification of $2$-dimensional toric quiver varieties (this result is known, see  Theorem 5.2 in \cite{hille:chemnitz} and Example 6.14 in \cite{fei}): 

\begin{proposition}\label{prop:twodim}
(i) A $2$-dimensional toric quiver variety is isomorphic to one of the following: 

\noindent The projective plane $\mathbb{P}^2$, the blow up of $\mathbb{P}^2$ in one, two, or three points in general position, or 
$\mathbb{P}^1\times\mathbb{P}^1$. 

(ii) The above varieties are realized (in the order of their listing) by the following quiver-weight pairs:  
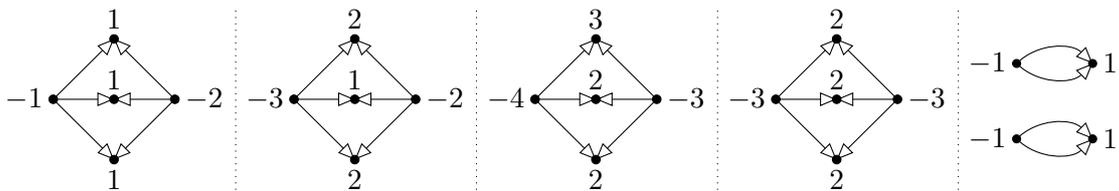
\[
\begin{tikzpicture}[>=open triangle 45,scale=0.8] 
\foreach \x in {(0,0),(1,0),(1,1),(1,-1),(2,0)} \filldraw \x circle (2pt); 
\foreach \y in {(1,0),(1,1),(1,-1)} \draw  [->]  (0,0)--\y; 
\foreach \x in {(1,0),(1,1),(1,-1)} \draw  [<-] \x--(2,0); 
\node [left] at (0,0) {$-1$};  \node [above] at (1,1) {$1$}; \node [above] at (1,0) {$1$}; \node [below] at (1,-1) {$1$}; \node [right] at (2,0) {$-2$}; 
\draw [dotted] (3,-1.5)--(3,1.5);
\end{tikzpicture}
\begin{tikzpicture}[>=open triangle 45,scale=0.8] 
\foreach \x in {(0,0),(1,0),(1,1),(1,-1),(2,0)} \filldraw \x circle (2pt); 
\foreach \y in {(1,0),(1,1),(1,-1)} \draw  [->]  (0,0)--\y; 
\foreach \x in {(1,0),(1,1),(1,-1)} \draw  [<-] \x--(2,0); 
\node [left] at (0,0) {$-3$};  \node [above] at (1,1) {$2$}; \node [above] at (1,0) {$1$}; \node [below] at (1,-1) {$2$}; \node [right] at (2,0) {$-2$}; 
\draw [dotted] (3,-1.5)--(3,1.5);
\end{tikzpicture}
\begin{tikzpicture}[>=open triangle 45,scale=0.8] 
\foreach \x in {(0,0),(1,0),(1,1),(1,-1),(2,0)} \filldraw \x circle (2pt); 
\foreach \y in {(1,0),(1,1),(1,-1)} \draw  [->]  (0,0)--\y; 
\foreach \x in {(1,0),(1,1),(1,-1)} \draw  [<-] \x--(2,0); 
\node [left] at (0,0) {$-4$};  \node [above] at (1,1) {$3$}; \node [above] at (1,0) {$2$}; \node [below] at (1,-1) {$2$}; \node [right] at (2,0) {$-3$}; 
\draw [dotted] (3,-1.5)--(3,1.5);
\end{tikzpicture} 
\begin{tikzpicture}[>=open triangle 45,scale=0.8] 
\foreach \x in {(0,0),(1,0),(1,1),(1,-1),(2,0)} \filldraw \x circle (2pt); 
\foreach \y in {(1,0),(1,1),(1,-1)} \draw  [->]  (0,0)--\y; 
\foreach \x in {(1,0),(1,1),(1,-1)} \draw  [<-] \x--(2,0); 
\node [left] at (0,0) {$-3$};  \node [above] at (1,1) {$2$}; \node [above] at (1,0) {$2$}; \node [below] at (1,-1) {$2$}; \node [right] at (2,0) {$-3$}; 
\draw [dotted] (3,-1.5)--(3,1.5);
\end{tikzpicture} 
\begin{tikzpicture}[>=open triangle 45]  \draw [->] (0,0.8) to [out=45, in=135] (1,0.8); \node [above] at (0,0) {}; \node [above] at (1,0) {}; 
\draw  [->] (0,0.8) to [out=315, in=225] (1,0.8) ;
\filldraw (0,0.8) circle (1.5pt) (1,0.8) circle (1.5pt); 
 \draw [->] (0,1.8) to [out=45, in=135] (1,1.8); 
\draw  [->] (0,1.8) to [out=315, in=225] (1,1.8) ;
\filldraw (0,1.8) circle (1.5pt) (1,1.8) circle (1.5pt); 
\node [left] at (0,0.8) {$-1$}; \node [left] at (0,1.8) {$-1$}; \node [right] at (1,0.8) {$1$}; \node [right] at (1,1.8) {$1$};
\end{tikzpicture}
\]
\end{proposition} 

\begin{proof} 
$\reducedquivers_1$ consists only of the Kronecker quiver. 
The only weights yielding a non-empty moduli space  are $(-1,1)$ and its positive integer multiples, hence the corresponding moduli space is $\mathbb{P}^1$. 
Thus $\mathbb{P}^1\times \mathbb{P}^1$, the product of two projective lines occurs as a $2$-dimensional toric quiver variety, say for 
the disjoint union of two copies of 
\begin{tikzpicture}[>=open triangle 45]  \draw [->] (0,0) to [out=45, in=135] (1,0); \node [above] at (0,0) {}; \node [above] at (1,0) {}; 
\draw  [->] (0,0) to [out=315, in=225] (1,0) ;
\filldraw (0,0) circle (1.5pt) (1,0) circle (1.5pt); 
\node [left] at (0,0) {$-1$};  \node [right] at (1,0) {$1$}; \
\end{tikzpicture}. 

$\sklist_2$ consists of the graph with  two vertices and  three edges connecting them (say by Proposition~\ref{prop:boundonskeletons}). 
Thus  $\reducedquivers'_2$ consists of the following quiver: 
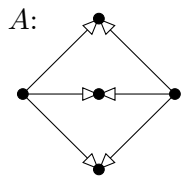
\[
\begin{tikzpicture}[>=open triangle 45] 
\foreach \x in {(0,0),(1,0),(1,1),(1,-1),(2,0)} \filldraw \x circle (2pt); 
\foreach \y in {(1,0),(1,1),(1,-1)} \draw  [->]  (0,0)--\y; 
\foreach \x in {(1,0),(1,1),(1,-1)} \draw  [<-] \x--(2,0);
\node at (0,1) {$A$:};
\end{tikzpicture}
\] 
Choosing a spanning tree $\subtree$ in $\quiver$,  the $x(a)$ with $a\in \quiver_1\setminus \subtree_1$ can be used as free coordinates in $\affspan(\polytope(\quiver,\theta))$. 
For example, take in the quiver $A$  the spanning tree with thick arrows in the following figure:  
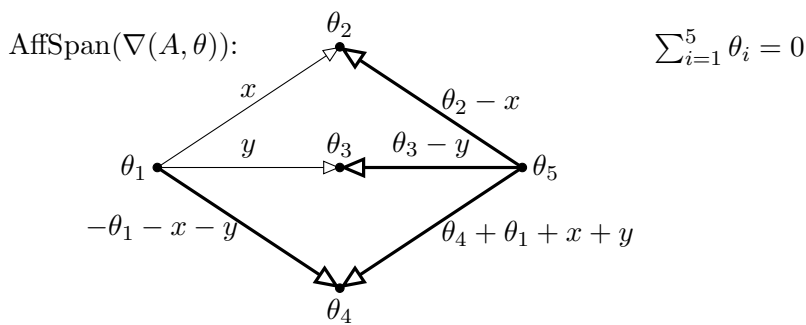
\[
\begin{tikzpicture}[>=open triangle 45,scale=0.8]
\node [left] at (1.5,2) {$\affspan(\polytope(A,\theta))$:}; 
\node [left] at (0,0) {$\theta_1$};   \node [above] at (3,2) {$\theta_2$}; \node [above] at (3,0) {$\theta_3$}; \node [below] at (3,-2) {$\theta_4$}; 
\node [right]  at (6,0) {$\theta_5$}; \node [right] at (8,2) {$ \sum_{i=1}^5\theta_i=0$}; 
\foreach \x in {(0,0),(3,0),(3,2),(3,-2),(6,0)} \filldraw \x circle (2pt); 
\foreach \y in {(3,0),(3,2)} \draw  [->]  (0,0)--\y;  \draw [->, very thick] (0,0)--(3,-2);
\foreach \x in {(3,0),(3,2),(3,-2)} \draw  [<-,very thick] \x--(6,0);
\node [above] at (1.5,1) {$x$}; \node [above] at (1.5,0) {$y$}; \node [left] at (1.5,-1) {$-\theta_1-x-y$};
\node [right] at (4.5,1.1) {$\theta_2-x$}; 
\node [above] at (4.5,0) {$\theta_3-y$}; \node [right] at (4.5,-1.1) {$\theta_4+\theta_1+x+y$};
\end{tikzpicture}
\]

\begin{figure}%[figure1] %\label{figure:ineqs}
\caption{{\it The polytope $\polytope(A,\theta)$}}
\begin{tikzpicture}[scale=0.7]
 \draw [->] (0,0)--(5,0); \draw [->] (0,0)--(0,5); \node [right] at (5,0) {$x$}; \node [above] at (0,5) {$y$}; \draw (-0.3,4.8)--(4.8,-0.3); 
 \node [above right] at (2.1,2.1) {$x+y=-\theta_1$};
\draw (-0.3,4)--(1,4); \node [left] at (-0.3,4) {$y=\theta_3$};
\draw (3.7,1.5)--(3.7,-0.3); \node [below] at (3.7,-0.3) {$x=\theta_2$}; 
\draw (-0.3,1.3)--(1.3,-0.3); \node [left] at (-0.3,1.3) {$x+y=-\theta_1-\theta_4$};
\draw [pattern=north east lines] (0,1)--(1,0)--(3.7,0)--(3.7,0.8)--(0.5,4)--(0,4)--(0,1);
\node  at (9,5) {Defining inequalitites:}; 
\node at (9,4) {$ 0 \le x\le \theta_2$};  
\node at (9,3) {$0 \le y\le \theta_3 $}; 
\node at (9,2) {$-\theta_4-\theta_1\le x+y\le -\theta_1$.};
\end{tikzpicture}
\end{figure}
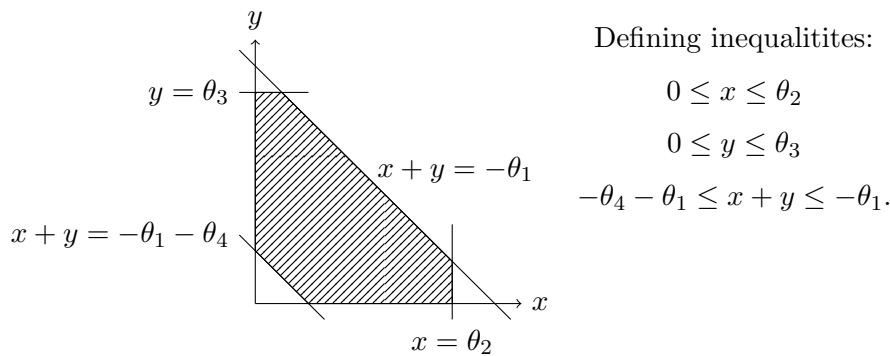 

Clearly $\polytope(A,\theta)$ is integral-affinely equivalent to the polytope in $\mathbb{R}^2=\{(x,y)\mid x,y\in \mathbb{R}\}$ shown on 
Figure~1. 
Depending on the order of $-\theta_1,\theta_3,-\theta_1-\theta_4,\theta_2$, its normal fan is one of the following: 
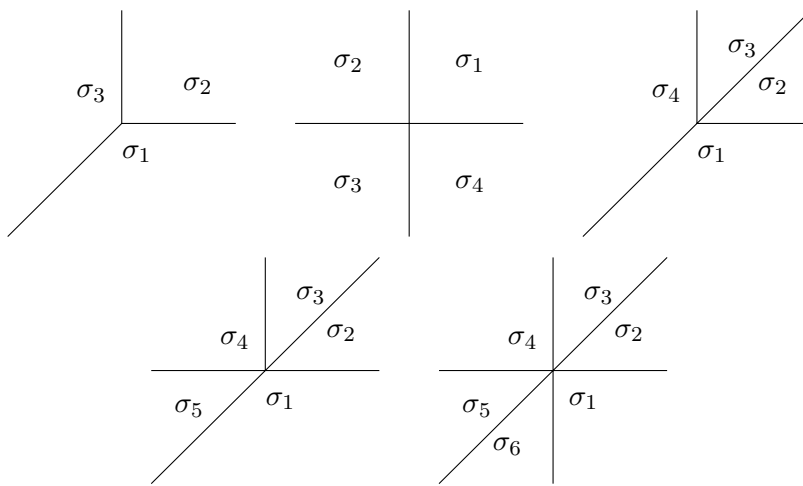
\[\begin{tikzpicture} \draw (0,0)--(1.5,0) (-1.5,-1.5)--(0,0) (0,0)--(0,1.5);
\node at (0.2,-0.4) {$\sigma_1$}; \node at (1,0.5) {$\sigma_2$}; \node at (-0.4,0.4) {$\sigma_3$}; 
\end{tikzpicture} 
\qquad
\begin{tikzpicture} \draw (0,0)--(1.5,0) (0,0)--(0,1.5) (-1.5,0)--(0,0) (0,-1.5)--(0,0);
\node at (0.8,0.8) {$\sigma_1$}; \node at (0.8,-0.8) {$\sigma_4$}; \node at (-0.8,-0.8) {$\sigma_3$}; \node at (-0.8,0.8) {$\sigma_2$}; 
\end{tikzpicture} 
\qquad 
\begin{tikzpicture} \draw (0,0)--(1.5,0) (0,0)--(0,1.5) (-1.5,-1.5)--(1.5,1.5);
\node at (0.2,-0.4) {$\sigma_1$}; \node at (1,0.5) {$\sigma_2$}; \node at (0.6,1) {$\sigma_3$}; \node at (-0.4,0.4) {$\sigma_4$}; 
\end{tikzpicture} 
\qquad 
\]
\[\begin{tikzpicture} \draw (-1.5,0)--(1.5,0) (0,0)--(0,1.5) (-1.5,-1.5)--(1.5,1.5);
\node at (0.2,-0.4) {$\sigma_1$}; \node at (1,0.5) {$\sigma_2$}; \node at (0.6,1) {$\sigma_3$}; \node at (-0.4,0.4) {$\sigma_4$}; 
\node at (-1,-0.5) {$\sigma_5$}; 
\end{tikzpicture} 
\qquad
\begin{tikzpicture} \draw (-1.5,0)--(1.5,0) (0,-1.5)--(0,1.5) (-1.5,-1.5)--(1.5,1.5);
\node at (0.4,-0.4) {$\sigma_1$}; \node at (1,0.5) {$\sigma_2$}; \node at (0.6,1) {$\sigma_3$}; \node at (-0.4,0.4) {$\sigma_4$}; 
\node at (-1,-0.5) {$\sigma_5$}; \node at (-0.6,-1) {$\sigma_6$}; 
\end{tikzpicture} 
\qquad 
\]
It is well known that the corresponding toric varieties are the projective plane $\mathbb{P}^2$, $\mathbb{P}^1\times \mathbb{P}^1$ and the projective plane blown up in one, two, or three points in general position, so (i) is proved. 
Taking into account the explicit inequalities in Figure 1, we see that for the pairs $(A,\theta)$ given in (ii), the variety $X_{\polytope(A,\theta)}=\moduli(A,\theta)$ has the desired isomorphism type. 
 \end{proof}

\begin{remark}\label{remark:nosmallerquiver} {\rm 
(i) Since the toric fan of the blow up of $\mathbb{P}^2$ in three generic points has $6$ rays, to realize it as a toric quiver variety we need a quiver with at least $6$ arrows and hence with at least $5$ vertices (see Proposition~\ref{prop:optimal}). 

(ii) Comparing Proposition~\ref{prop:twodim} with Section 3.3 in \cite{altmann-nill-schwentner-wiercinska} we conclude that for each isomorphism class of a $2$-dimensional toric quiver variety there is a quiver $\quiver$ such that  $\moduli(\quiver,\delta_{\quiver})$ belongs to the given isomorphism class (recall that $\delta_{\quiver}$ is the so-called canonical weight), in particular  in dimension $2$ every projective toric quiver variety is Gorenstein Fano.. This is explained by the following two facts: (1) in dimension $2$, a complete fan is determined by the set of rays; (2) if $(\quiver,\theta)$ is tight, then $(\quiver,\delta_{\quiver})$ is tight. Now (1) and (2) imply that if $(\quiver,\theta)$ is tight and $\chi(\quiver)=2$, then $\moduli(\quiver,\theta)\cong\moduli (\quiver,\delta_{\quiver})$. 

(iii) The above does not hold in dimension three or higher. Consider for example the quiver-weight pairs:
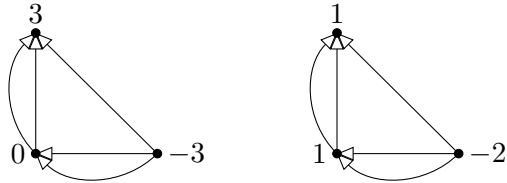
\[
\begin{tikzpicture}[>=open triangle 45,scale=0.8] 
\foreach \x in {(0,0),(2,0),(0,2)} \filldraw \x circle (2pt); 
\draw  [->]  (2,0)--(0,0);
\draw  [->]  (2,0) to [out=225,in=315] (0,0);
\draw  [->]  (0,0)--(0,2);
\draw  [->]  (0,0) to [out=135,in=225] (0,2);
\draw  [->]  (2,0)--(0,2);
\node [left] at (0,0) {$0$};  \node [above] at (0,2) {$3$}; \node [right] at (2,0) {$-3$}; 

\end{tikzpicture}
\hspace*{3em}{\begin{tikzpicture}[>=open triangle 45,scale=0.8] 
\foreach \x in {(2,0),(4,0),(2,2)} \filldraw \x circle (2pt); 
\draw  [->]  (4,0)--(2,0);
\draw  [->]  (4,0) to [out=225,in=315] (2,0);
\draw  [->]  (2,0)--(2,2);
\draw  [->]  (2,0) to [out=135,in=225] (2,2);
\draw  [->]  (4,0)--(2,2);
\node [left] at (2,0) {$1$};  \node [above] at (2,2) {$1$}; \node [right] at (4,0) {$-2$}; 
\end{tikzpicture}}
\]

The weight on the left is the canonical weight $\delta_{\quiver}$ for this quiver, and it is easy to check that $(\quiver,\delta_{\quiver})$ is tight and $\moduli(\quiver,\delta_{\quiver})$ is a Gorenstein Fano variety with one singular point. The weight on the right is also tight for this quiver, however it gives a smooth moduli space which can not be isomorphic to $\moduli(\quiver,\delta_{\quiver})$, consequently it also can not be Gorenstein Fano since the rays in its fan are the same as those in the fan of $\moduli(\quiver,\delta_{\quiver})$.

(iv) It is also notable in dimension $2$ that each toric moduli space can be realized by precisely one quiver from $\reducedquivers_d$. This does not hold in higher dimensions. For example consider the following quivers:
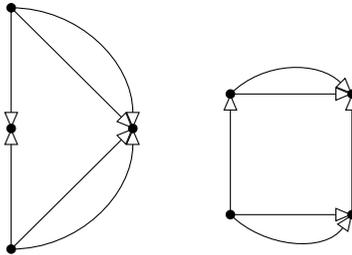
\[
\begin{tikzpicture}[>=open triangle 45,scale=0.8] 
\foreach \x in {(0,0),(0,2),(0,-2),(2,0)} \filldraw \x circle (2pt); 
\draw  [->]  (0,2)--(0,0);
\draw  [->]  (0,-2)--(0,0);
\draw  [->]  (0,2)--(2,0);
\draw  [->]  (0,2) to [out=0,in=90] (2,0);
\draw  [->]  (0,-2)--(2,0);
\draw  [->]  (0,-2) to [out=0,in=270] (2,0);
\end{tikzpicture}
\hspace*{3em}{\begin{tikzpicture}[>=open triangle 45,scale=0.8] 
\foreach \x in {(0,-1),(0,1),(2,-1),(2,1)} \filldraw \x circle (2pt); 
\draw  [->]  (0,-1)--(0,1);
\draw  [->]  (2,-1)--(2,1);
\draw  [->]  (0,-1)--(2,-1);
\draw  [->]  (0,-1) to [out=-45,in=235] (2,-1);
\draw  [->]  (0,1)--(2,1);
\draw  [->]  (0,1) to [out=45,in=135] (2,1); 
\end{tikzpicture}}
\]

These quivers are both tight with their canonical weights, and give isomorpic moduli, since they are both obtained after tightening:
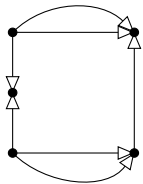
\[
\begin{tikzpicture}[>=open triangle 45,scale=0.8] 
\foreach \x in {(0,0),(0,1),(0,-1),(2,1),(2,-1)} \filldraw \x circle (2pt); 
\draw  [->]  (0,1)--(0,0);
\draw  [->]  (0,-1)--(0,0);
\draw  [->]  (0,1)--(2,1);
\draw  [->]  (0,1) to [out=45,in=135] (2,1);
\draw  [->]  (0,-1)--(2,-1);
\draw  [->]  (0,-1) to [out=315,in=235] (2,-1);
\draw  [->]  (2,-1)--(2,1);
\end{tikzpicture}
\]
}
\end{remark}

%%%%%%%%%%%%%%%%%%%%%%%%%%%%%%%

\section{Affine quotients} \label{sec:affine} 

We need a result concerning representation spaces that we discuss for general dimension vectors. 
Consider the following situation. Let $\subforest$ be a (not necessarily full) subquiver of $\quiver$ which is the disjoint union of trees $\subforest=\coprod_{i=1}^r\subtree^i$ 
(where by a {\it tree} we mean a quiver whose underlying graph is a tree). 
Let  $\alpha$ be a dimension vector 
taking the same value $d_i$ on the vertices of each $\subtree^i$ $(i=1,\dots,r)$.  Let $\theta\in\mz^{\quiver_0}$ be a weight such that there exist positive integers $n_a$ $(a\in \subforest_1)$ with $\theta(v)=\sum_{a\in\subforest_1:a^+=v}n_a-\sum_{a\in\subforest_1:a^-=v}n_a$. The representation space $\rep(\quiver,\alpha)$ contains the Zariski dense open subset 
\[U_{\subforest}:=\{R\in\rep(\quiver,\alpha)\mid \forall a\in\subforest_1:\ \det(R(a))\neq 0\}.\] 
Note that $U_{\subforest}$ is a principal affine open subset in $\rep(\quiver,\alpha)$ given by the non-vanishing of the relative invariant 
$f:R\mapsto \prod_{a\in\subforest_1}\det^{n_a}(R(a))$  of weight $\theta$, hence $U_{\subforest}$ is contained in $\rep(\quiver,\alpha)^{\theta-ss}$. 
Moreover, $U_{\subforest}$ is $\pi$-saturated with respect to the quotient morphism $\pi:\rep(\quiver,\alpha)^{\theta-ss}\to \moduli(\quiver,\alpha,\theta)$, 
hence $\pi$ maps $U_{\subforest}$ onto an open subset $\pi(U_{\subforest})\cong U_{\subforest}/\!/\! GL(\alpha)$ of $\moduli(\quiver,\alpha,\theta)$ (here for an affine $GL(\alpha)$-variety $X$ we denote by $X/\!/\! GL(\alpha)$ the 
affine quotient, that is, the variety with coordinate ring the ring of invariants $\coord(X)^{GL(\alpha)}$), see \cite{newstead}.   
Denote by $\hat\quiver$ the quiver obtained from $\quiver$ by contracting each connected component $\subtree^i$ of $\subforest$ to a single vertex $t_i$ $(i=1,\dots,r)$. 
So $\hat\quiver_0=\quiver_0\setminus \subforest_0\coprod \{t_1,\dots,t_r\}$ and its arrow set can be identified with $\quiver_1\setminus\subforest_1$, but if an end vertex of an arrow belongs to $\subtree^i$ in $\quiver$ then viewed as an arrow in $\hat\quiver$ the correspoding end vertex is $t_i$ (in particular, an arrow in $\quiver_1\setminus\subforest_1$ connecting two vertices of $\subtree^i$ becomes a loop at vertex $t_i$).  Denote by $\hat\alpha$ the dimension vector obtained by contracting $\alpha$ accordingly, so $\hat\alpha(t_i)=d_i$ for $i=1,\dots,r$ and $\hat\alpha(v)=\alpha(v)$ for $v\in\hat\quiver_0\setminus\{t_1,\dots,t_r\}$. 
Sometimes we shall identify $GL(\hat\alpha)$ with the subgroup of $GL(\alpha)$ consisting of the elements $g\in GL(\alpha)$ with the property that 
$g(v)=g(w)$ whenever $v,w$ belong to the same component $\subtree^i$ of $\subforest$. 
We have a $GL(\hat\alpha)$-equivariant embedding 
\begin{equation}\label{eq:iota}\iota:\rep(\hat\quiver,\hat\alpha)\to \rep(\quiver,\alpha)\end{equation}
defined by $\iota(x)(a)=x(a)$ for $a\in\hat\quiver_1$ and $\iota(x)(a)$ the identity matrix for $a\in\quiver_1\setminus \hat\quiver_1$. Clearly $\mathrm{Im}(\iota)\subseteq \rep(\quiver,\alpha)^{\theta-ss}$. 

\begin{proposition}\label{prop:slices} 
\begin{itemize}
\item[(i)] $U_{\subforest}\cong GL(\alpha)\times_{GL(\hat\alpha)}\rep(\hat\quiver,\hat\alpha)$ 
as affine $GL(\alpha)$-varieties. 
\item[(ii)] The map $\iota$ induces an isomorphism $\bar\iota: \moduli(\hat\quiver,\hat\alpha,0)\stackrel{\cong} \longrightarrow \pi(U_{\subforest})\subseteq \moduli(\quiver,\alpha,\theta)$. 
 \end{itemize} 
\end{proposition}
 
\begin{proof} (i) Set $p:=\iota(0)\in\rep(\quiver,\alpha)$. Clearly $GL(\hat\alpha)$ is the stabilizer of $p$ in $GL(\alpha)$ acting on $\rep(\quiver,\alpha)$, hence the $GL(\alpha)$-orbit $O$ of $p$ is isomorphic to 
$GL(\alpha)/GL(\hat\alpha)$ via the map sending the coset $gGL(\hat\alpha)$ to $g\cdot p$. 
On the other hand $O$ is the subset  consisting of all those points $R\in\rep(\quiver,\alpha)$ for which 
$\det(R(a))\neq 0$ for $a\in\subforest_1$ and $R(a)=0$ for all $a\notin \subforest_1$.  This can be shown by induction on the number of arrows of $\subforest$, using  the assumption that $\subforest$ is the disjoint union of trees.  Recall also that the arrow set of $\hat\quiver$ is identified with a subset   $\quiver_1\setminus\subforest_1$. This yields an obvious identification $U_{\subforest}=\rep(\hat\quiver,\hat\alpha)\times O$. Projection $\varphi:U_{\subforest}\to O$ onto the second component is $GL(\alpha)$-equivariant by construction. 
Moreover, the fibre $\varphi^{-1}(p)=\iota(\rep(\hat\quiver,\hat\alpha))\cong \rep(\hat\quiver,\hat\alpha)$ as a $\mathrm{Stab}_{GL(\alpha)}(p)=GL(\hat\alpha)$-varieties. 
It is well known that this implies the isomorphism $U_{\subforest} \cong GL(\alpha)\times_{GL(\hat\alpha)}\rep(\hat\quiver,\hat\alpha)$, see for example Lemma 5.17 in \cite{bongartz}. 

(ii) It follows from (i) that $U_{\subforest} /\!/\! GL(\alpha)\cong \rep(\hat\quiver,\hat\alpha)/\!/\! GL(\hat\alpha)=\moduli(\hat\quiver,\hat\alpha,0)$ by standard properties of associated fiber products. Furthermore, 
taking into account the proof of (i) we see 
$U_{\subforest}/\!/\! GL(\alpha)=\pi(\varphi^{-1}(p))=\pi(\iota(\rep(\hat\quiver,\hat\alpha))$ where $\pi$ is the quotient morphism \eqref{eq:quotient}.  
 \end{proof} 
 
 Let us apply Proposition~\ref{prop:slices} in the toric case. 
It is well known that for a  lattice point $m$ in a lattice polyhedron $\polytope$  there is an affine open toric subvariety $U_m$ of $X_{\polytope}$, and $X_{\polytope}$ is covered by these affine open subsets as $m$ ranges over the set of vertices of $\polytope$  (see Section 2.3 in \cite{cox-little-schenck}). For a toric quiver variety realized as $\mathrm{Im}(\rho)$ as in \eqref{eq:rho} and Proposition~\ref{prop:realizationofmoduli}, this can be seen explicitly as follows: $U_{m_i}$ is the complement in $\mathrm{Im}(\rho)$ of the affine hyperplane $\{(x_0:\dots:x_d)\mid x_i=0\}\subseteq \mathbb{P}^d$, 
and for a general lattice point $m$ in the quiver polyhedron, $U_m$ is the intersection of finitely many $U_{m_i}$.

A subset $S$ of $\quiver_0$ is {\it successor closed} if for any $a\in\quiver_1$ with $a^-\in S$ we have $a^+\in S$. A subquiver $\quiver'$ of $\quiver$ is {\it $\theta$-stable} if 
$\theta(\quiver'_0)=0$ and for any non-empty $S\subsetneq \quiver'_0$ which is successor closed in $\quiver'$ we have that $\theta(S)>0$. 
The {\it support} of $x\in\mz^{\quiver_1}$ is the quiver with vertex set $\support(x)_0:=\quiver_0$ and arrow set $\support(x)_1:=\{a\in\quiver_1\mid x(a)\neq 0\}$. Now $m\in \polytope(\quiver,\theta)$ is a vertex if and only if the connected components of $\support(m)$ are $\theta$-stable subtrees of $\quiver$. On the other hand for each subquiver $T$ of $\quiver$ that is the disjoint union of $\theta$-stable subtrees and satisfies $T_0=\quiver_0$ there is  precisely one vertex $m\in \polytope(\quiver,\theta)$ such that $\support(m)=T$ (see for example Corollary 8 in \cite{altmann-straten}). 
Given a vertex $m$ of the polyhedron $\polytope(\quiver,\theta)$ denote by $(\quiver^m,\theta^m)$ the quiver and weight obtained by  successively contracting the arrows in 
$\support(m)$. Clearly $\theta^m$ is the zero weight. 
The following statement can be viewed as a stronger version for the toric case of the results \cite{adriaenssens-lebruyn} on the {\it local quiver settings} of a moduli space of quiver representations: the \'etale morphisms used for general dimension vectors in  \cite{adriaenssens-lebruyn} can be replaced by isomorphisms in the toric case.

\begin{theorem}\label{thm:vertices} 
For any vertex $m$ of the quiver polyhedron $\polytope(\quiver,\theta)$ the affine open toric subvariety $U_m$ in $\moduli(\quiver,\theta)$ is isomorphic to 
$\moduli(\quiver^m,0)$. 
Moreover, $\iota:\rep(\quiver^m)\to \rep(\quiver)$ defined as in \eqref{eq:iota} induces an isomorphism 
$\bar\iota:\moduli(\quiver^m,0)\stackrel{\cong}\longrightarrow U_m\subseteq \moduli(\quiver,\theta)$. 
\end{theorem} 

\begin{proof} This is a special case of Proposition~\ref{prop:slices} (ii). 
\end{proof} 

Conversely, any affine toric quiver variety $\moduli(\quiver',0)$ can be obtained as $U_m\subseteq \moduli(\quiver,\theta)$ for some projective toric quiver variety 
$\moduli(\quiver,\theta)$ and a vertex $m$ of the quiver polytope $\polytope(\quiver,\theta)$. In fact we have a more general result, which is a refinement for the toric case of 
Theorem 2.2 in \cite{domokos:gmj}: 

\begin{theorem}\label{thm:compactification} 
For any quiver polyhedron $\polytope(\quiver,\theta)$ with $k$ vertices there exists a bipartite quiver $\tilde\quiver$, a weight $\theta' \in\mz^{\tilde\quiver_1}$, and a set 
$m_1,\dots,m_k$ of vertices of the quiver polytope $\polytope(\tilde\quiver,\theta')$ such  that the quasiprojective toric variety 
$\moduli(\quiver,\theta)$ is isomorphic to the open subvariety $\bigcup_{i=1}^kU_{m_i}$ of the projective toric quiver variety $\moduli(\tilde\quiver,\theta')$. 
\end{theorem}

\begin{proof} Double the quiver $\quiver$ to get a bipartite quiver $\tilde\quiver$ as follows: to each $v\in\quiver_0$ there corresponds a source $v_-$ and a sink $v_+$ in 
$\tilde\quiver$, for each $a\in\quiver_1$ there is an arrow in $\tilde\quiver$ denoted by the same symbol $a$, such that if $a\in\quiver_1$ goes from $v$ to $w$, then 
$a\in\tilde\quiver_1$ goes from $v_-$ to $w_+$, and for each $v\in\quiver_0$ there is a new arrow $e_v\in\tilde\quiver_1$ from $v_-$ to $v_+$. 
Denote by $\tilde\theta\in\mz^{\tilde\quiver_0}$ the weight $\tilde\theta(v_-)=0$, $\tilde\theta(v_+)=\theta(v)$, and set $\kappa\in\mz^{\tilde\quiver_0}$ with 
$\kappa(v_-)=-1$ and $\kappa(v_+)=1$ for all $v\in\quiver_0$.  

Suppose that $\subforest$ is a $\theta$-stable subtree in $\quiver$. Denote by $\tilde\subforest$ the subquiver of $\tilde\quiver$ consisting of the arrows with the same label as the arrows of $\subforest$, in addition to the arrows $e_v$ for each $v\in \subforest_0$. It is clear that $\tilde\subforest$ is a subtree of $\tilde\quiver$. We claim that 
$\tilde\subforest$ is $(\tilde\theta+d\kappa)$-stable for sufficiently large $d$. Obviously $(\tilde\theta+d\kappa)(\tilde\subforest_0)=0$. Let $\tilde S$ be a proper successor closed subset of $\tilde\subforest_0$ in $\tilde\quiver$. Denote by $S\subset \subforest_0$ consisting of $v\in\subforest_0$ with $v_+\in\tilde S$ (note that $v_-\in S$ implies $v_+\in S$, since $e_v\in \tilde\subforest$). We have the equality 
$(\tilde\theta+d\kappa)(\tilde S)=\theta(S)+\sum_{v^+\in\tilde S,v_-\notin \tilde S}(\theta(v)+d)$. 
If the second summand is the empty sum (i.e. $v_+\in\tilde S$ implies $v_-\in\tilde S$), then $S$ is successor closed, hence $\theta(S)>0$ by assumption. Otherwise the sum is positive for sufficiently large $d$. This proves the claim. It follows that if $d$ is sufficiently large, then for any vertex $m$ of $\polytope(\quiver,\theta)$, setting 
$\subforest:=\support(m)$, there exists a vertex $\tilde m$ of $\polytope(\tilde\quiver,\tilde\theta+d\kappa)$ with $\support(\tilde m)=\tilde\subforest$. 

Denote by $\mu:\rep(\quiver)\to\rep(\tilde\quiver)$ the map defined by $\mu(x)(e_v)=1$ for each $v\in\quiver_0$, and $\mu(x)(a)=x(a)$ for all $a\in\tilde\quiver_1$. 
This is equivariant, where we identify $(\mc^\times)^{\quiver_0}$ with the stabilizer of $\mu(0)$  in $(\mc^{\times})^{\tilde\quiver_0}$. 
The above considerations show that $\mu(\rep(\quiver)^{\theta-ss})\subseteq \rep(\quiver)^{(\tilde\theta+d\kappa)-ss}$, whence $\mu$ induces a morphism 
$\bar\mu:\moduli(\quiver,\theta)\to\moduli(\tilde\quiver,\tilde\theta+d\kappa)$. 
Restrict $\bar\mu$ to the affine open subset $U_m\subseteq\moduli(\quiver,\theta)$, and compose $\bar\mu\vert_{U_m}$ with the isomorphism 
$\bar\iota:\moduli(\quiver^m,0)\to U_m\subseteq \moduli(\quiver,\theta)$ from Theorem~\ref{thm:vertices}. 
By construction we see that $\bar\mu\vert_{U^m}\circ\bar\iota$ is the isomorphism $\moduli(\quiver^m,0)\to U_{\tilde m}$ of Theorem~\ref{thm:vertices}. 
It follows that $\bar\mu\vert_{U_m}:U_m\to U_{\tilde m}$ is an isomorphism. As $m$ ranges over the vertices of $\polytope(\quiver,\theta)$, these isomorphisms 
glue together to the isomorphism $\bar\mu:\moduli(\quiver,\theta)\to \bigcup_{\tilde m}U_{\tilde m}\subseteq \moduli(\tilde\quiver,\tilde\theta)$.  
\end{proof} 
 
 We note that similarly to Theorem 2.2 in \cite{domokos:gmj}, it is possible to embed $\moduli(\quiver,\theta)$ as an open subvariety into 
 a projective variety $\moduli(\tilde\quiver,\theta')$, such that for any vertex $m'$ of $\polytope(\tilde\quiver,\theta')$ the affine open subvariety $U_{m'}\subseteq \moduli(\tilde\quiver,\theta')$ is isomorphic to $U_m\subseteq \moduli(\quiver,\theta)$ for some vertex $m$ of $\polytope(\quiver,\theta)$ (but of course typically 
 $\polytope(\tilde\quiver,\theta')$ has more vertices than $\polytope(\quiver,\theta)$). In particular, a smooth variety $\moduli(\quiver,\theta)$ can be embedded into 
 a smooth projective toric quiver variety $\polytope(\tilde\quiver,\theta')$, where $\tilde\quiver$ is bipartite. 
 
 %%%%%%%%%%%%%%%%%%%%%%%%%%%%%%%%%%

\section{Classifying affine toric quiver varieties}\label{sec:affine-class} 
 
In this section we deal with the zero weight. It is well known and easy to see (say by Remark~\ref{remark:combinatorial-tightness}) that $\quiver$ is $0$-stable if and only if 
$\quiver$ is {\it strongly connected}, that is, for any ordered pair $v,w\in\quiver_0$ there is an oriented path in $\quiver$ from $v$ to $w$. 

\begin{proposition}\label{prop:valencyfour} 
Let $\quiver$ be a prime quiver with $\chi(\quiver)\ge 2$, such that $(\quiver,0)$ is tight. 
\begin{itemize}
\item[(i)] For any $v\in \quiver_0$ we have  
$|\{a\in\quiver_1\mid a^-=v\}|\ge 2$ and $|\{a\in\quiver_1\mid a^+=v\}|\ge 2$.  
\item[(ii)] We have $|\quiver_0|\le \chi(\quiver)-1$ (and consequently $|\quiver_1|=|\quiver_0|+\chi(\quiver)-1\le 2(\chi(\quiver)-1)$. 
\end{itemize}
\end{proposition} 

\begin{proof} (i) Suppose $v\in\quiver_0$ and $a\in\quiver_1$ is the only arrow with $a^-=v$. Then the equations \eqref{eq:flow} imply that for any $x\in\polytope(\quiver,0)$ we have $x(a)=\sum_{b^+=v}x(b)$, so by Lemma~\ref{lemma:contractable} the arrow $a$ is contractable. The case when $a$ is the only arrow with $a^+=v$ is similar. 

(ii) By (i) the valency of any vertex is at least $4$, hence similar considerations as in the proof  of  Proposition~\ref{prop:boundonskeletons} yield the desired bound 
on $|\quiver_0|$. 
\end{proof} 

Denote by $\reducedquivers''_d$ the set of prime quivers $\quiver$ with $\chi(\quiver)=d$ and $(\quiver,0)$ tight. 
Then $\reducedquivers''_1$ consists of the one-vertex quiver with a single loop, $\reducedquivers''_2$ is empty, 
$\reducedquivers''_3$ consists of the quiver with $2$ vertices and $2-2$ arrows in both directons. 
$\reducedquivers''_4$ consists of three quivers:

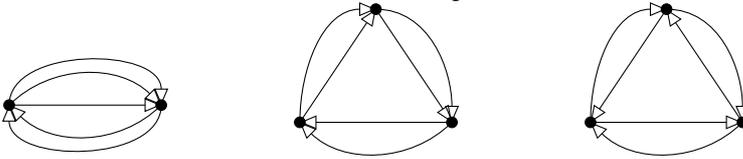
\begin{tikzpicture}[>=open triangle 45]  
\foreach \x in {(0,0),(2,0)} \filldraw \x circle (2pt); 
\draw [->] (0,0) to [out=90,in=90] (2,0); \draw [->] (0,0) to [out=45,in=135] (2,0); \draw [->] (0,0) to (2,0); 
\draw [<-] (0,0) to [out=270,in=270] (2,0);  \draw [<-] (0,0) to [out=315,in=225] (2,0); 
\end{tikzpicture}
\qquad\qquad
\begin{tikzpicture}[>=open triangle 45]  
\foreach \x in {(0,0),(1,1.5),(2,0)} \filldraw \x circle (2pt); 
\draw[->] (0,0) to (1,1.5); \draw[->] (1,1.5) to (2,0); \draw[->] (2,0) to (0,0); 
\draw[->] (0,0) to [out=90,in=180] (1,1.5); \draw[->] (1,1.5) to [out=0,in=90] (2,0); \draw [<-] (0,0) to [out=315,in=225] (2,0);
\end{tikzpicture} 
\qquad\qquad
\begin{tikzpicture}[>=open triangle 45]  
\foreach \x in {(0,0),(1,1.5),(2,0)} \filldraw \x circle (2pt); 
\draw[<-] (0,0) to (1,1.5); \draw[<-] (1,1.5) to (2,0); \draw[<-] (2,0) to (0,0); 
\draw[->] (0,0) to [out=90,in=180] (1,1.5); \draw[->] (1,1.5) to [out=0,in=90] (2,0); \draw [<-] (0,0) to [out=315,in=225] (2,0);
\end{tikzpicture} 

\begin{example}\label{example:affinedegree} {\rm 
Consider the quiver $\quiver$ with $d$ vertices and $2d$ arrows $a_1,\dots,a_d,b_1,\dots,b_d$, where $a_1\dots a_d$ is a primitive cycle and $b_i$ is the obtained by reversing $a_i$ for $i=1,\dots,d$.  Then $\chi(\quiver)=d+1$, and after the removal of any of the arrows of $\quiver$ we are left with a strongly connected quiver. 
So $(\quiver,0)$ is tight, showing that the bound in Proposition~\ref{prop:valencyfour} (ii) is sharp. The coordinate ring $\coord(\moduli(\quiver,0))$ is the subalgebra of 
$\coord(\rep(\quiver))$ generated by $\{x(a_i)x(b_i),x(a_1)\cdots x(a_d), x(b_1)\cdots x(b_d)\mid i=1,\dots,d\}$, so it is the factor ring of the $(d+2)$-variable polynomial ring 
$\mc[t_1,\dots,t_{d+2}]$ modulo the ideal generated by $t_1\cdots t_d-t_{d+1}t_{d+2}$.  }
\end{example}

%%%%%%%%%%%%%%%%%%%%%%%%%%%%%%%%%%%%%%%%%%%%%%%%%%%

 \section{Presentations of semigroup algebras}\label{sec:semigroupalgebras} 
 
 Let $\quiver$ be  a quiver with no oriented cycles, $\theta\in\mz^{\quiver_0}$ a weight. For $a\in\quiver_1$ denote by $x(a):R\mapsto R(a)$ the corresponding coordinate function on $\rep(\quiver)$, and for a lattice point $m\in\polytope(\quiver,\theta)$ set 
 $x^m:=\prod_{a\in\quiver_1}x(a)^{m(a)}$. The homogeneous coordinate ring $\homcoord(\quiver,\theta)$ of $\moduli(\quiver,\theta)$ is the subalgebra of 
 $\coord(\rep(\quiver))$ generated by $x^m$  where $m$ ranges over $\polytope(\quiver,\theta)\cap\mz^{\quiver_1}$. 
 In Section~\ref{sec:equations} we shall study  {\it the ideal of relations} among the generators $x^m$.  This leads us to the context of presentations of polytopal semigroup algebras  (cf. section 2.2 in \cite{bruns-gubeladze}),  since 
 $\homcoord(\quiver,\theta)$ is naturally identified with the semigroup algebra $\mc[\semigr(\quiver,\theta)]$, where 
 \[\semigr(\quiver,\theta):=\coprod_{k=0}^\infty \polytope(\quiver,k\theta)\cap\mz^{\quiver_1}.\] 
This is a submonoid  of $\mz^{\quiver_1}$. By normality of the polytope $\polytope(\quiver,\theta)$ it is the same as the submonoid of $\mn_0^{\quiver_1}$ generated by $\polytope(\quiver,\theta)\cap\mz^{\quiver_1}$.    
 
First we formulate a statement  (Lemma~\ref{lemma:joo}; a version of it was introduced in \cite{joo}) in a slightly more general situation than what is needed here.  
Let $\semigr$ be any finitely generated commutative monoid (written additively) with non-zero generators $s_1,\dots,s_d$, and denote by $\mz[\semigr]$ the corresponding semigroup algebra over $\mz$: its elements are formal integral linear combinations of the symbols $\{x^s\mid s\in\semigr\}$, with multiplication given by $x^s\cdot x^{s'}=x^{s+s'}$. 
Write $R:=\mz[t_1,\dots,t_d]$ for the $d$-variable polynomial ring  over the integers, and 
$\phi:R \to \mz[\semigr]$ the ring surjection $t_i\mapsto x^{s_i}$. Set $I:=\ker(\phi)$. It is well known and easy to see that 
\begin{equation}\label{eq:binomspan}I=\ker(\phi)=\mathrm{Span}_{\mz}\{t^a-t^b\mid \sum_{i=1}^da_is_i=\sum_{j=1}^db_js_j\in\semigr\}\end{equation} 
where for $a=(a_1,\dots,a_d)\in\mn_0^d$ we write $t^a=t_1^{a_1}\dots t_d^{a_d}$. 

Introduce a binary relation on the set of monomials in $R$: we write $t^a\sim t^b$ if $t^a-t^b\in R_+I$, where $R_+$ is the ideal in $R$ consisting of the  polynomials 
with zero constant term. 
Obviously $\sim$ is an equivalence relation. Let $\Lambda$ be a complete set of representatives of the equivalence classes. 
We have $\Lambda=\coprod_{s\in \semigr}\Lambda_s$, where 
for $s\in\semigr$ set $\Lambda_s:=\{t^a\in\Lambda\mid \sum a_is_i=s\}$. For the $s\in \semigr$ with $|\Lambda_s|>1$, set 
$\mathcal{G}_s:=\{t^{a_1}-t^{a_i}\mid i=2,\dots,p\}$, where $t^{a_1},\dots,t^{a_p}$ is an arbitrarily chosen ordering of the elements of $\Lambda_s$. 

\begin{lemma}\label{lemma:mingen}
Suppose that $\semigr=\coprod_{k=0}^\infty\semigr_k$ is graded (i.e. $\semigr_k\semigr_l\subseteq \semigr_{k+l}$) and   $\semigr_0=\{0\}$ (i.e. the generators $s_1,\dots,s_d$ have positive degree). 
Then $\coprod_{s\in\semigr:|\Lambda_s|>1}\mathcal{G}_s$ is a minimal homogeneous generating system of the ideal $I$, where the grading on $\mz[t_1,\dots,t_d]$ is defined by setting the degree of $t_i$ to be equal to the  degree of $s_i$. 
In particular, $I$ is minimally generated by $\sum_{s\in\semigr}(|\Lambda_s|-1)$ elements. 
\end{lemma} 

\begin{proof} It is easy to see that a $\mz$-module direct complement of $R_+I$ in $R$ is $\sum_{t^a\in\Lambda}\mz t^a$. Thus the statement follows by the graded Nakayama Lemma. 
\end{proof} 

Next for a cancellative commutative monoid $\semigr$ we give a more explicit  description of the relation $\sim$ (a special case occurs in \cite{joo}). 
For some elements $s,v\in\semigr$ we say that $s$ {\it divides} $v$ and write $s\mid v$ if there exists an element $w\in\semigr$ with $v=s+w$. 
For any $s\in\semigr$ introduce a binary relation $\sim_s$ on the subset of $\{s_1,\dots,s_d\}$ consisting of the generators $s_i$ with $s_i\mid s$ as follows: 
\begin{eqnarray}\label{eq:joo} 
s_i\sim_s s_j \mbox{ if }i=j\mbox{ or there exist }u_1,\dots,u_k\in\{s_1,\dots,s_d\}
\\ \notag \mbox{ with }u_1=s_i,\ u_k=s_j,\ u_lu_{l+1}\mid s\mbox{ for }l=1,\dots,k-1.
\end{eqnarray} 
Obviously $\sim_s$ is an equivalence relation, and $s_i\sim_s s_j$ implies $s_i\sim_t s_j$ for any $s\mid t\in\semigr$. 

\begin{lemma}\label{lemma:joo} Let $\semigr$ be a cancellative commutative monoid generated by $s_1,\dots,s_d$. 
Take $t^a-t^b\in I$, so  $s:=\sum_{i=1}^da_is_i=\sum_{j=1}^db_js_j\in\semigr$. Then the following are equivalent: 
 \begin{itemize} 
\item[(i)] $t^a-t^b\in R_+I$;  
\item[(ii)] For some $t_i\mid t^a$ and  $t_j\mid t^b$ we have $s_i\sim_s s_j$; 
\item[(iii)] For all  $t_i\mid t^a$ and  $t_j\mid t^b$ we have $s_i\sim_s s_j$.  
\end{itemize} 
\end{lemma} 

\begin{proof} (iii) trivially implies (ii). Moreover, if $t_i,t_j$ are two different variables occuring in $t^a$ with $\sum a_is_i=s$, then $s_is_j\mid s$, hence 
taking $k=2$ and $u_1=s_1$, $u_2=s_2$ in \eqref{eq:joo} we see that  $s_i\sim_s s_j$. This shows that (ii) implies (iii). 

To show that (ii)  implies (i)  assume  that for some $t_i\mid t^a$ and $t_j\mid t^b$ we have $s_i\sim_s s_j$.  If $s_i=s_j$, then $t^a$ and $t^b$ have a common variable, say  $t_1$,  so $t^a=t_1t^{a'}$ and $t^b=t_1t^{b'}$ for some 
$a',b'\in\mn_0^d$. We have 
\[x^{s_1}\phi(t^{a'}-t^{b'})=\phi(t_1(t^{a'}-t^{b'}))=\phi(t^a-t^b)=0\]
hence $x^{s_1}\phi(t^{a'})=x^{s_1}\phi(t^{b'})$. Since $\semigr$ is cancellative, we conclude $\phi(t^{a'})=\phi(t^{b'})$, thus $t^{a'}-t^{b'}\in I$, implying in turn that 
$t^a-t^b=t_1(t^{a'}-t^{b'})\in R_+I$. 
If $s_i\neq s_j$, then there exist 
$z_1,\dots,z_k\in\{t_1,\dots,t_d\}$ such that $u_l\in\semigr$ with $\phi(z_l)=x^{u_l}$ satisfy \eqref{eq:joo}. 
Then there exist monomials (possibly empty) $w_0,\dots,w_k$ in the variables $t_1,\dots.,t_d$ such that 
\[z_1w_0=t^a,\quad \phi(z_lz_{l+1}w_l)=x^s\quad  (l=1,\dots,k-1), \quad z_kw_k=t^b.\] 
 It follows that 
 \begin{equation}\label{eq:3binomspan} t^a-t^b=z_1(w_0-z_2w_1)+\sum_{l=2}^{k-1}z_l(z_{l-1}w_{l-1}-z_{l+1}w_l)+z_k(z_{k-1}w_{k-1}-w_k).\end{equation}
 Note that $\phi(z_1w_0-z_1z_2w_1)=x^s-x^s=0$. Hence $z_1(w_0-z_2w_1)\in I$.  Since $\semigr$ is cancellative, we conclude that $z_1(w_0-z_2w_1)\in R_+I$. 
 Similarly all the other summands on the right hand side of  \eqref{eq:3binomspan} belong to $R_+I$, hence $t^a-t^b\in R_+I$. 

Finally we show that (i) implies (ii). 
Suppose that $t^a-t^b\in R_+I$. By \eqref{eq:binomspan} we have 
\begin{equation}\label{eq:2binomspan} 
t^a-t^b=\sum_{l=1}^k t_{i_l}(t^{a_l}-t^{b_l})\ \mbox{ where }\ t^{a_l}-t^{b_l}\in I\mbox{ and }i_l\in\{1,\dots,d\}\end{equation} 
After a possible renumbering we may assume that 
\begin{equation}\label{eq:cancellations}t_{i_1}t^{a_1}=t^a, \ t_{i_l}t^{b_l}=t_{i_{l+1}}t^{a_{l+1}}\  \mbox{ for }\  l=1,\dots,k-1, \ \mbox{ and }\ t_{i_k}t^{b_k}=t^b.\end{equation} 
Observe that if $t_{i_l}=t_{i_{l+1}}$ for some $l\in\{1,\dots,k-1\}$, then necessarily $t^{b_l}=t^{a_{l+1}}$, hence  
$t_{i_l}(t^{a_l}-t^{b_l})+t_{i_{l+1}}(t^{a_{l+1}}-t^{b_{l+1}})=t_{i_l}(t^{a_l}-t^{b_{l+1}})$. 
Thus in \eqref{eq:2binomspan} we may replace the sum of the $l$th and $(l+1)$st terms by a single summand $t_{i_l}(t^{a_l}-t^{b_{l+1}})$. 
In other words, we may achieve that in \eqref{eq:2binomspan} we have $t_{i_l}\neq t_{i_{l+1}}$ for each $l=1,\dots,k-1$, in addition to \eqref{eq:cancellations}. 
If $k=1$, then $t^a$ and $t^b$ have a common variable and (ii) obviously holds. From now on assume that $k\ge 2$. 
From $t_{i_l}t^{a_l}=t_{i_{l+1}}t^{b_l}$ and the fact that $t_{i_l}$ and $t_{i_{l+1}}$  are different variables in $\mz[t_1,\dots,t_d]$ we deduce that 
$t^{a_l}=t_{i_{l+1}}t^{c_l}$ for some $c_l\in\mn_0^d$, implying that 
$x^s=\phi(t_{i_l}t^{a_l})=\phi(t_{i_l}t_{i_{l+1}}t^{c_l})=\phi(t_{i_l})\phi(t_{i_{l+1}})\phi(t^{c_l})$. Thus $u_l:=s_{i_l}$ satisfy \eqref{eq:joo} and  hence $s_{i_1}\sim_s s_{i_k}$. 
\end{proof}

\begin{corollary}\label{cor:joo} 
Suppose that $\semigr=\coprod_{k=0}^\infty\semigr_k$ is a finitely generated graded cancellative commutative monoid generated by $\semigr_1=\{s_1,\dots,s_d\}$.  
The kernel of $\phi:\mz[t_1,\dots,t_d] \to \mz[\semigr]$, $t_i\mapsto x^{s_i}$ is generated by homogeneous elements of degree at most $r$ (with respect to the standard grading on $\mz[t_a,\dots,t_d]$) if and only if for all $k>r$ and 
$s\in\semigr_k$, the elements in 
$\semigr_1$ that divide $s$ in the monoid $\semigr$ form a single equivalence class with respect to  $\sim_s$. 
\end{corollary} 

\begin{proof} This is an immediate consequence of Lemma~\ref{lemma:mingen} and Lemma~\ref{lemma:joo}.  
\end{proof} 

%%%%%%%%%%%%%%%%%%%%%%%%%%%%%%%%%%%%%%%%%
\section{Equations of toric quiver varieties}\label{sec:equations}

Corollary~\ref{cor:joo} applies for the monoid $\semigr(\quiver,\theta)$, where the grading is given by 
$\semigr(\quiver,\theta)_k=\polytope(\quiver,k\theta)\cap\mz^{\quiver_1}$. 
Recall that we may identify the complex semigroup algebra $\mc[\semigr(\quiver,\theta)]$ and the homogeneous coordinate ring 
$\homcoord(\quiver,\theta)$ by identifying the basis element $x^m$ in the semigroup algebra to the element of $\homcoord(\quiver,\theta)$ denoted by the same symbol $x^m$. 
Introduce a variable $t_m$ for each $m\in \polytope(\quiver,\theta)\cap\mz^{\quiver_1}$, take the polynomial ring 
\[F:=\mc[t_m\mid m\in \polytope(\quiver,\theta)\cap\mz^{\quiver_1} ]\] 
and consider the surjection 
\begin{equation}\label{eq:varphi}  \varphi:F\to\homcoord(\quiver,\theta), \qquad t_m\mapsto x^m.\end{equation}  
The kernel $\ker(\varphi)$ is a homogeneous ideal in the polynomial ring $F$ (endowed with the standard grading) called  {\it the ideal of relations among the $x^m$}, for which Corollary~\ref{cor:joo} applies. 
Note also that in the monoid $\semigr(\quiver,\theta)$ we have that $m\mid n$ for some $m,n$ if and only if $m\le n$, 
where the partial ordering $\le$ on $\mz^{\quiver_1}$ is defined by setting $m\le n$ if $m(a)\le n(a)$ for all  $a\in\quiver_1$. 
The following statement is a special case of the main result (Theorem 2.1) of \cite{yamaguchi-ogawa-takemura}: 

\begin{proposition}\label{prop:theta1}
Let $\quiver=K(n,n)$ be the complete bipartite quiver with $n$ sources and $n$ sinks, with a single arrow from each source to each sink. Let $\theta$ be the weight with $\theta(v)=-1$ for each source and $\theta(v)=1$ for each sink,  and   $\varphi:F \to \homcoord(Q,\theta)$ given in \eqref{eq:varphi}. 
Then the ideal $\ker(\varphi)$ is generated by elements of degree at most $3$.
\end{proposition}

For sake of completeness we present a proof. The argument below is based on the key idea of \cite{yamaguchi-ogawa-takemura}, but we use a different language and obtain a very short derivation of the result. 
For this quiver and weight generators of $\homcoord(\quiver,\theta)$ correspond to perfect matchings of the underlying graph of $K(n,n)$. Recall that a {\it perfect matching} 
of $K(n,n)$ is a set of  arrows $\{a_1,\dots,a_n\}$ such that  
for each source  $v$ there is a unique $i$ such that $a_i^-=v$ and for each sink $w$ there is a unique $j$ such that $a_j^+=w$. 
Now $\polytope(\quiver,\theta)\cap\mz^{\quiver_1}$ in this case consists of the characteristic functions of perfect matchings of $K(n,n)$. By a \textit{near perfect matching} we mean an incomplete matching that covers all but $2$ vertices ($1$ sink and $1$ source). Abusing language we shall freely identify a (near) perfect matching and its characteristic function (an element of $\mn_0^{\quiver_1}$). 
First we show the following lemma:

\begin{lemma} \label{lemma:nearperfect}
Let $\theta$ be the weight for $\quiver=K(n,n)$ as above, and $m_1+\dots+m_k=q_1+\dots+q_k$ for some  $k\ge 4$ and $m_i,q_j \in \polytope(\quiver,\theta)\cap\mz^{\quiver_1}$. Furthermore let us assume that for some $0\leq l \leq  n-2$ there is a near perfect matching $p$ such that $p \leq m_1 + m_2$ and $p$ contains $l$ arrows from $q_1$. Then there is a $j \geq 3$ and $m'_1,m'_2,m'_j\in \polytope(\quiver,\theta)\cap\mz^{\quiver_1}$ and a near perfect matching $p'$ such that $m_1 + m_2 + m_j=m'_1 + m'_2 + m'_j$, $p' \leq m'_1 + m'_2$ and $p'$ contains $l + 1$ arrows from $q_1$.
\end{lemma}

\begin{proof}
Let $v_1, \dots, v_n$ be the sources and $w_1, \dots, w_n$  the sinks of $\quiver$, and let us assume that $p$ covers all vertices but $v_1$ and $w_1$. Let $a$ be the arrow incident to $v_1$ in $q_1$. If $a$ is contained in $m_1 + m_2$ then pick an arbitrary $j \geq 3$, otherwise take $j$ to be such that $m_j$ contains $a$. We can obtain a near perfect matching $p' < m_1 + m_2 + m_j$ that intersects $q_1$ in $l + 1$ arrows in the following way: if $a$ connects $v_1$ and $w_1$ we add $a$ to $p$ and remove one arrow from it that was not contained in $q_1$ (this is possible due to $l \leq n-2$); if $a$ connects $v_1$ and $w_i$ for some $i \neq 1$ then we add $a$ to $p$ and remove the arrow from $p$ which was incident to $w_i$ (this arrow is not contained in $q_1$). 
Set $r:=m_1+m_2+m_j-p'\in\mn_0^{\quiver_1}$, and denote by $S$ the subquiver of $K(n,n)$ with $S_0=\quiver_0$ and $S_1=\{c\in\quiver_1\mid r(c)\ne 0\}$.  
We have $S_0=S_0^-\coprod S_0^+$ where  $S_0^-$ denotes the set of sources and $S_0^+$ denotes the set of sinks. For a vertex $v\in S_0$ set 
$\deg_r(v):= \sum_{c\in S_1} |r(c)|$. 
We have that  $\deg_r(v)=3$ for exactly one source and for exactly one sink, and  $\deg_r(v)=2$ for all the remaining vertices of $S$. Now let 
$A$ be an arbitrary subset of $S_0^-$, and denote by $B$ the subset of $S_0^+$ consisting of the sinks that are connected by an arrow in $S$ to a vertex in $A$. 
We have the inequality $\sum_{v\in A}\deg_r(v)\le \sum_{w\in B} \deg_r(w)$. Since on both sides of this inequality  the summands are $2$ or $3$, and $3$ can occur at most once on each side, we conclude that $|B|\ge |A|$.  Applying the K\"onig-Hall Theorem (cf. Theorem 16.7 in \cite{schrijver}) to $S$ we conclude that it contains a perfect matching. Denote  the characteristic vector of this perfect matching by $m'_j$.  Take perfect matchings  $m'_1$ and $m'_2$ of $S$ with $m_1 + m_2 + m_j - m'_j=m_1'+m_2'$ 
(note that $m_1',m_2'$ exist by normality of the polytope $\polytope(\quiver,\theta)$, which in this case can be seen as an imediate consequence of the K\"onig-Hall Theorem). By construction we have $m_1 + m_2 + m_j=m'_1 + m'_2 + m'_j$, $p' \leq m'_1 + m'_2$, and $p'$ has $l+1$ common arrows with $q_1$. 
\end{proof}

\begin{proofofprop}~\ref{prop:theta1}  By Corollary~\ref{cor:joo} it is sufficient to show that 
if $s=m_1+\cdots+m_k=q_1+\cdots+q_k$ where $m_i,q_j\in \polytope(\quiver,\theta)\cap\mz^{\quiver_1}$ and $k\ge 4$, then the $m_i,q_j$ all belong to the same equivalence class with respect to $\sim_s$. 
Note that since $k\ge 4$, the elements $m_1',m_2',m_j'$ from the statement of Lemma~\ref{lemma:nearperfect} belong to the same equivalence class with respect to $\sim_s$ as  $m_1,\dots,m_k$.  
Hence repeatedly applying Lemma~\ref{lemma:nearperfect} we may assume  that there is a near perfect matching $p\le m_1+m_2$ 
such that $p$ and $q_1$ have $n - 1$ common arrows. The only arrow of $q_1$ not belonging to $p$ belongs to some $m_j$, hence after a possible renumbering of 
$m_3,\dots,m_k$ we may assume that $q_1\le m_1+m_2+m_3$. It follows that $q_1\sim_s m_4$, implying in turn that the $m_i,q_j$ all belong to the same quivalence class with respect to $\sim_s$. 
\end{proofofprop}

Now we are in position to state and prove the main result of this section (this was stated in \cite{lenz} as well, but was withdrawn later, see \cite{lenz-withdrawn}):  

\begin{theorem}\label{thm:degreethree}  
Let $\quiver$ be a quiver with no oriented cycles, $\theta\in\mz^{\quiver_1}$ a weight, and $\varphi$ the $\mc$-algebra surjection given in \eqref{eq:varphi}. Then the ideal 
$\ker(\varphi)$ is generated by elements of degree at most $3$. 
\end{theorem} 

\begin{proof} 
By Proposition~\ref{prop:addanarrow} and the double quiver construction (cf. the proof of Theorem~\ref{thm:compactification}) it is sufficient to deal with the case when $\quiver$ is bipartite and $\polytope(\quiver,\theta)$ is non-empty. 
This implies that $\theta(v)\le 0$ for each source vertex $v$ and $\theta(w)\ge 0$ for each sink vertex $w$. Note that  if $\theta(v)=0$  for some vertex $v\in\quiver_0$, 
then omitting $v$ and the arrows adjacent to $v$ we get a quiver $\quiver'$ such that the lattice polytope $\polytope(\quiver,\theta)$ is integral-affinely equivalent to 
$\polytope(\quiver',\theta\vert_{\quiver'_0})$, hence we may assume that $\theta(v)\neq 0$ for each $v\in\quiver_0$. 
We shall apply induction on $\sum_{v\in\quiver_0}(|\theta(v)|-1)$. 

The induction starts with the case when $\sum_{v\in\quiver_0}(|\theta(v)|-1)=0$, in other words, $\theta(v)=-1$ for each source $v$ and $\theta(w)=1$ for each sink $w$. 
This forces that 
the number of sources equals to the number of sinks in $\quiver$. 
The case when $\quiver$ is the complete bipartite quiver $K(n,n)$ having $n$ sinks and $n$ sources, and each source is connected to each sink by a single arrow is covered by Proposition~\ref{prop:theta1}.  
Suppose next that $\quiver$ is a subquiver of $K(n,n)$ having a relative invariant of weight $\theta$ (i.e. $K(n,n)$ has a perfect matching all of whose arrows belong to 
$\quiver$).  The lattice polytope $\polytope(\quiver,\theta)$ can be identified with a subset of $\polytope(K(n,n),\theta)$: 
think of $m\in\mz^{\quiver_1}$ as $\tilde{m}\in\mz^{K(n,n)_1}$ where $\tilde{m}(a)=0$ for $a\in K(n,n)_1\setminus\quiver_1$ and $\tilde{m}(a)=m(a)$ for $a\in\quiver_1\subseteq K(n,n)_1$. The  surjection $\tilde{\varphi}:\mc[t_m\mid m\in\polytope(K(n,n),\theta)]\to\homcoord(K(n,n),\theta)$ 
restricts to $\varphi:\mc[t_m\mid m\in\polytope(\quiver,\theta)]\to\homcoord(\quiver,\theta)$. Denote by $\pi$ the surjection of polynomial rings that sends to zero 
the variables $t_m$ with $m\notin \polytope(\quiver,\theta)$. Then $\pi$ maps the ideal $\ker(\tilde{\varphi})$ onto $\ker(\varphi)$, consequently 
generators of $\ker(\tilde{\varphi})$ are mapped onto generators of $\ker(\varphi)$. Since we know already that the first ideal is generated by elements of degree at most $3$, the same holds for $\ker(\varphi)$. 
The case when $\quiver$ is an arbitrary bipartite quiver with $n$ sources and $n$ sinks having possibly multiple arrows, and $\theta(v)=-1$ for each source $v$ and 
$\theta(w)=1$ for each sink $w$ follows from the above case by a repeated application of Proposition~\ref{prop:multiplearrow} below. 

Assume next that $\sum_{v\in\quiver_0}(|\theta(v)|-1)\ge 1$, so there exists a vertex $w\in \quiver_0$ with $|\theta(w)|>1$. By symmetry we may assume that $w$ is a sink, so 
$\theta(w)>1$. Construct a new quiver $\quiver'$ as follows: add a new vertex $w'$ to $\quiver_0$, for each arrow $b$ with $b^+=w$ add an extra arrow $b'$ with 
$(b')^+=w'$ and $(b')^-=b^-$, and consider the weight $\theta'$ with $\theta'(w')=1$, $\theta'(w)=\theta(w)-1$, and $\theta'(v)=\theta(v)$ for all other vertices $v$. 
By Corollary~\ref{cor:joo} it is sufficient to show that if 
\[m_1+\dots+m_k=n_1+\dots+n_k=s\in \semigr:=\semigr(\quiver,\theta)\] 
for some $k\ge 4$ and $m_1,\dots,m_k,n_1,\dots,n_k \in \polytope(\quiver,\theta)\cap\mz^{\quiver_1}$, 
then   $m_i\sim_s n_j$ for some (and hence all) $i,j$.  
Set $\semigr':=\semigr(\quiver',\theta')$, and consider the semigroup homomorphism 
$\pi:\semigr'\to\semigr$ given by 
\[\pi(m')(a)=\begin{cases} m'(a)+m'(a')&\quad \mbox{ if } a^+=w;\\m'(a)&\quad \mbox{ if }a^+\neq w.\end{cases}\]
Take an arrow $\alpha$ with $\alpha^+=w$ and $s(\alpha)>0$. After a possible renumbering we may assume that $m_1(\alpha)>0$ and $n_1(\alpha)>0$. 
Define $m_1'\in\mn_0^{\quiver'_1}$ as 
$m_1'(\alpha)=m_1(\alpha)-1$, $m_1'(\alpha')=1$, and $m_1'(a)=m_1(a)$ for all other arrows $a\in\quiver'_1$. 
Similarly define $n_1'\in\mn_0^{\quiver'_1}$ as $n_1'(\alpha)=n_1(\alpha)-1$, $n_1'(\alpha')=1$, and $n_1'(a)=n_1(a)$ for all other arrows $a\in\quiver'_1$. 
Clearly  $\pi(m_1')=m_1$, $\pi(n_1')=n_1$. 
Now we construct $s'\in\semigr'$ with $\pi(s')=s$,   $s'-m_1'\in\mn_0^{\quiver'_1}$ and $s'-n_1'\in\mn_0^{\quiver'_1}$ (thus $m_1'$ and $n_1'$ divide  $s'$ in $\semigr'$). 
Note that $\sum_{a^+=w}s(a)=k\theta(w)$ and $\sum_{a^+=w} \max\{m_1(a),n_1(a)\}<\sum_{a^+=w}(m_1(a)+n_1(a))=2\theta(w)$ (since $m_1(\alpha)>0$ and $n_1(\alpha)>0$). The inequalities $\theta(w)\ge 2$ and $k\ge 3$ imply that $\sum_{a^+=w}(s(a)-\max\{m_1(a),n_1(a)\})\ge k$. Consequently there exist 
non-negative integers $\{t(a)\mid a^+=w \}$ such that $\sum_{a^+=w}t(a)=(\sum_{a^+=w}s(a))-k$, $s(a)\ge t(a)\ge \max\{m_1(a),n_1(a)\}$ for all $a\neq \alpha$ with $a^+=w$, and $s(\alpha)-1\ge t(\alpha)\ge \max\{m_1(\alpha),n_1(\alpha)\}-1$. Consider 
$s'\in \mz^{\quiver'_1}$ given by $s'(a')=s(a)-t(a)$ and $s'(a)=t(a)$ for $a\in \quiver_1$ with $a^+=w$ and $s'(b)=s(b)$ for all other $b\in\quiver'_1$. 
By construction $s'$ has the desired properties, and 
so there exist $m_i',n_j'\in\polytope(\quiver',\theta')$ with $s'=m_1'+\dots+m_k'=n_1'+\dots+n_k'$. 
Since $\sum_{v\in\quiver'_0}(|\theta'(v)|-1)$ is one less than $\sum_{v\in\quiver_0}(|\theta(v)|-1)$, by the induction hypothesis we have 
$m_1'\sim_{s'}n_1'$. It is clear that $a\sim_t b$ implies $\pi(a)\sim_{\pi(t)} \pi(b)$, so we deduce $m_1\sim_s n_1$. As we pointed out before, this shows by 
Corollary~\ref{cor:joo} that $\ker(\varphi)$ is generated by elements of degree at most $3$. 
\end{proof}

The above proof refered to a general recipe to derive a minimal generating system of $\ker(\varphi)$ from a minimal generating system for the quiver obtained by collapsing multiple arrows to a single arrow. Let us consider the following situation: let $\quiver$ be a quiver with no oriented cycles, $\alpha_1,\alpha_2\in\quiver_1$ with $\alpha_1^-=\alpha_2^-$ and $\alpha_1^+=\alpha_2^+$. Denote by $\quiver'$ the quiver obtained from $\quiver$ by collapsing the  $\alpha_i$ to a single arrow $\alpha$. 
Take a weight $\theta\in\mz^{\quiver_0}=\mz^{\quiver'_0}$. 
The map $\pi:\polytope(\quiver,\theta)\to\polytope(\quiver',\theta)$ mapping $m\mapsto m'$ with 
$m'(\alpha)=m(\alpha_1)+m(\alpha_2)$ and $m'(\beta)=m(\beta)$ for all $\beta\in\quiver'_1\setminus\{\alpha\}=\quiver_1\setminus\{\alpha_1,\alpha_2\}$ 
induces a surjection 
from the monoid $\semigr:=\semigr(\quiver,\theta)$  onto the monoid  
$\semigr':=\semigr(\quiver',\theta')$. This extends to a surjection of semigroup algebras $\pi:\mc[\semigr]\to\mc[\semigr']$, which are 
identified with $\homcoord(\quiver,\theta)$ and $\homcoord(\quiver',\theta)$, respectively. Keep the notation $\pi$ for the induced $\mc$-algebra surjection 
 $\homcoord(\quiver,\theta)\to \homcoord(\quiver',\theta)$. We have the commutative diagram of $\mc$-algebra surjections 
 
 \[\begin{array}{ccc}
F=\mc[t_m\mid m\in \polytope(\quiver,\theta)\cap\mz^{\quiver_1}] & \stackrel{\varphi }\longrightarrow & \homcoord(\quiver,\theta) \\
\downarrow{{\scriptstyle{\pi}}} &  &\downarrow{{\scriptstyle{\pi}}} \\
F'=\mc[t_{m'}\mid m'\in \polytope(\quiver',\theta)\cap\mz^{\quiver'_1}] & \stackrel{\varphi'}\longrightarrow & \homcoord(\quiver',\theta)
\end{array}\]
where the left vertical map (denoted also by $\pi$) sends the variable $t_m$ to $t_{\pi(m)}$.  
For any monomial $u \in F'$ and any $s\in \semigr$ with $\pi(x^s)=\varphi'(u)\in\semigr'$ we choose a monomial 
$\psi_s(u)\in F$  such that $\pi(\psi_s(u))=u$ and $\varphi(\psi_s(u))=x^s$. This is clearly possible: let $u=t_{m_1}\dots t_{m_r}$, then 
we take for $\psi_s(u)$ an element $t_{n_1}\dots t_{n_r}$ where $\pi(n_j)=m_j$, such that $(n_1+\dots+n_r)(\alpha_1)=s(\alpha_1)$.  
Denote by $\varepsilon_i\in\mn_0^{\quiver_1}$ the characteristic function of $\alpha_i\in\quiver_1$ $(i=1,2)$. 

\begin{proposition}\label{prop:multiplearrow} 
Let $u_{\lambda}-v_{\lambda}$ $(\lambda\in\Lambda)$ be a set of binomial relations generating the ideal $\ker(\varphi')$. 
Then $\ker(\varphi)$ is generated by $\mathcal{G}_1\bigcup\mathcal{G}_2$, where 
\begin{eqnarray*}\mathcal{G}_1&:=& \{\psi_s(u_{\lambda})-\psi_s(v_{\lambda})\mid \lambda\in\Lambda, \pi(s)=\varphi'(u_{\lambda})\}  \\ 
\mathcal{G}_2&:=& 
 \{t_mt_n-t_{m+\varepsilon_2-\varepsilon_1}t_{n+\varepsilon_1-\varepsilon_2}\mid m,n\in\polytope(\quiver,\theta)\cap\mz^{\quiver_1}, m(\alpha_1)>0, n(\alpha_2)>0\}.
 \end{eqnarray*}
\end{proposition} 

\begin{proof} Clearly $\mathcal{G}_1$ and $\mathcal{G}_2$ are contained in $\ker(\varphi)$. Denote by $I$ the ideal generated by them in $F$, so $I\subseteq \ker(\varphi)$.  
In order to show the reverse inclusion, take any binomial relation $u-v\in\ker(\varphi)$, then $\varphi(u)=\varphi(v)=x^s$ for some $s\in\semigr$. It follows that $\pi(u)-\pi(v)\in\ker(\varphi')$, whence there exist monomials $w_i$ such that $\pi(u)-\pi(v)=\sum_{i=1}^kw_i(u_i-v_i)$, where 
$u_i-v_i\in \{u_{\lambda}-v_{\lambda}, \quad v_{\lambda}-u_{\lambda}\mid \lambda\in\Lambda\}$, 
$w_1u_1=\pi(u)$, $w_iv_i=w_{i+1}u_{i+1}$ for $i=1,\dots,k-1$ and $w_kv_k=\pi(v)$. 
Moreover, for each $i$ choose a divisor $s_i\mid s$ such that $\pi(x^{s_i})=\varphi'(u_i)$ (this is clearly possible). 
Then $I$ contains the element 
$\sum_{i=1}^k\psi_{s-s_i}(w_i)(\psi_{s_i}(u_i)-\psi_{s_i}(v_i))$, whose $i$th summand we shall denote by $y_i-z_i$ for notational simplicity. 
Then we have that $\pi(y_1)=\pi(u)$, $\pi(z_k)=\pi(v)$,  $\pi(z_i)=\pi(y_{i+1})$ for $i=1,\dots,k-1$, and $x^s=\varphi(y_i)=\varphi(z_i)$. 
It follows by Lemma~\ref{lemma:quadratic} below $u-y_1$, $v-z_k$,  and $y_{i+1}-z_i$ for $i=1,\dots,k-1$  are all contained in the ideal $J$ generated by $\mathcal{G}_2$. 
Whence $u-v$ is contained in $I$. 
\end{proof} 

\begin{lemma}\label{lemma:quadratic} Suppose that for monomials $u,v\in F$ 
we have $\varphi(u)=\varphi(v)\in \homcoord(\quiver,\theta)$ and $\pi(u)=\pi(v)\in F'$. 
Then $u-v$ is contained in the ideal $J$ generated by $\mathcal{G}_2$ (with the notation of Proposition~\ref{prop:multiplearrow}). 
\end{lemma} 

\begin{proof} If $u$ and $v$ have a common variable $t$, then $u-v=t(u'-v')$, and $u',v'$ satisfy the conditions of the lemma. 
By induction on the degree we may assume that $u'-v'$ belongs to the ideal $J$. 
Take $m_1\in\polytope(\quiver,\theta)\cap\mz^{\quiver_1}$ such that $t_{m_1}$ is a variable occurring in $u$. 
There exists an $m_2\in\polytope(\quiver,\theta)\cap\mz^{\quiver_1}$ such that $t_{m_2}$ occurs in $v$, and $\pi(m_1)=\pi(m_2)$. 
By symmetry we may assume that $m_1(\alpha_1)\ge m_2(\alpha_1)$, and apply induction on the non-negative difference 
$m_1(\alpha_1)-m_2(\alpha_1)$. If $m_1(\alpha_1)-m_2(\alpha_1)=0$, then $m_1=m_2$ and we are done by the above considerations. 
Suppose next that $m_1(\alpha_1)-m_2(\alpha_1)>0$. 
By $\pi(m_1)=\pi(m_2)$ we have $m_2(\alpha_2)>0$, and 
the condition $\varphi(u)=\varphi(v)$ implies that 
there exists an 
$m_3\in \polytope(\quiver,\theta)\cap\mz^{\quiver_1}$ such that $t_{m_2}t_{m_3}$ divides $v$, and $m_3(\alpha_1)>0$. 
Denote  by $\varepsilon_i\in\mn_0^{\quiver_1}$ the characteristic function of $\alpha_i$, and set 
$m_2':=m_2+\varepsilon_1-\varepsilon_2$, $m_3':=m_3-\varepsilon_1+\varepsilon_2$. Clearly 
$m_2',m_3'\in\polytope(\quiver,\theta)\cap\mz^{\quiver_1}$ and $t_{m_2}t_{m_3}-t_{m_2'}t_{m_3'}\in J$. So modulo $J$ we may replace 
$v$ by $t_{m_2'}t_{m_3'}v'$ where $v=t_{m_2}t_{m_3}v'$. 
Clearly $0\le m_1(\alpha_1)-m_2'(\alpha_1)<m_1(\alpha_1)-m_2(\alpha_1)$, and by induction  we are finished.  
\end{proof} 

In the affine case one can also introduce a grading on  $\coord(\moduli(\quiver,0))$ by declaring the elements that correspond to primitve cycles of the quiver to be of degree $1$. The ideal of relations can be defined as above, but in this case it is not possible to give a degree bound independently of the dimension. 
This is illustrated by Example~\ref{example:affinedegree}, providing an instance where a degree  $d-1$ element is needed to generate the ideal of relations of a $d$-dimensional affine toric quiver variety. However the following theorem shows that this example is the worst possible from this respect.

\begin{theorem}\label{thm:affinebound} 
Let $\quiver$ be a quiver such that $d := \dim(\moduli(\quiver,0)) >0$. Then the ideal of relations of $\moduli(\quiver,0)$ is generated by elements of degree at most $d-1$.
\end{theorem}

\begin {proof} Up to dimension $2$ the only affine toric quiver varieties are the affine spaces. Suppose from now on that $d\ge 3$. 
Clearly it is sufficient to deal with the case when $(\quiver,0)$ is tight and $\quiver$ is prime. 
Suppose that a degree $k$ element is needed  to generate the ideal of relations of $\moduli(\quiver,0)$. 
In Section 6 of \cite{joo} it is shown that this holds  if and only if there is a pair of primitive cycles $c_1$, $c_2$ in $\quiver$ such that the multiset sum of their arrows can also be obtained as the multiset sum of some other $k$ primitive cycles $e_1,\dots,e_k$. 
Note that each $e_i$ has an arrow contained in $c_1$ but not in $c_2$, and has an arrow contained in $c_2$ but not in $c_1$.  It follows that 
$\mathrm{length}(c_1)+\mathrm{length}(c_2)\ge 2k$, implying that $\quiver$ has 
at least $k$ vertices.  By  Proposition~\ref{prop:valencyfour} (ii) we conclude that $d-1=\chi(\quiver)-1\geq |\quiver_0|\ge k$. 
 \end{proof}
 
 %%%%%%%%%%%%%%%%%%%%%%%%%%%%%%%%%%
 
 \section{The general case in \cite{yamaguchi-ogawa-takemura}}\label{sec:japanok}

\def\osmpoly{OSM(\quiver)} 

In this section we give a short derivation of the main result of 
 \cite{yamaguchi-ogawa-takemura} from the special case Proposition~\ref{prop:theta1}.
To reformulate the result in our context consider a bipartite quiver $\quiver$ with at least as many sinks as sources. By a \textit{one-sided matching} of $\quiver$ we mean an 
arrow set which has exactly one arrow incident to each source, and at most one arrow incident to each sink.  By abuse of language the characteristic vector in $\mz^{\quiver_1}$ of a one-sided matching will also be called a one-sided matching. The convex hull of the one-sided matchings in $\mz^{\quiver_1}$ is a lattice polytope in $\mr^{\quiver_1}$ which we will denote by $\osmpoly$. Clearly the lattice points of $\osmpoly$ are precisely the one-sided matchings. The normality of $\osmpoly$ is explained in section 4.2 of \cite{yamaguchi-ogawa-takemura} or it can be directly shown using the K\"onig-Hall Theorem for regular graphs and an argument similar to that in the proof below. 
Denote by $\semigr(\osmpoly)$ the submonoid of $\mn_0^{\quiver_1}$ generated by $\osmpoly\cap\mz^{\quiver_1}$. This is graded, the generators have degree $1$. 
Consider  the ideal of relations among the generators $\{x^m\mid m\in\osmpoly\cap\mz^{\quiver_1}\}$ of the semigroup algebra $\mc[\semigr(\osmpoly)]$. Theorem 2.1 from  \cite{yamaguchi-ogawa-takemura} can be stated as follows:

\begin{theorem}
The ideal of relations of $\mc[\semigr(\osmpoly)]$ is generated by binomials of degree at most 3.
\end{theorem}

\begin{proof}
Consider a quiver $\quiver'$ that we obtain by adding enough new sources to $Q$ so that it has the same number of sources and sinks, and adding an arrow from each new source to every sink. Let $\theta$ be the weight of $Q'$ that is $-1$ on each source and $1$ on each sink. Now the natural projection $\pi : \mr^{Q'_1} \rightarrow \mr^{Q_1}$ induces a surjective map from $\polytope(\quiver',\theta)\cap\mz^{Q'_1}$ onto $\osmpoly\cap\mz^{Q_1}$ giving us a degree preserving surjection between the corresponding semigroup algebras. 
By Corollary~\ref{cor:joo} 
it is sufficient to prove that for any $k \geq 4$, any degree $k$ element $s\in \semigr(\osmpoly)$, and any $m,n \in \osmpoly \cap\mz^{\quiver_1}$ with 
$m,n$ dividing $s$ we have $m\sim_s n$.  In order to show this we shall construct  an $s' \in \polytope(\quiver',k\theta)\cap\mz^{\quiver'_1}$ and $m' , n' \in  \polytope(\quiver',\theta)\cap\mz^{\quiver'_1}$ such that $m'\le s'$, $n' \leq s'$, $\pi(m') = m$, $\pi(n') = n$ and $\pi(s') = s$. By Proposition~\ref{prop:theta1} we have $m' \sim_{s'} n'$, hence the surjection $\pi$ yields $m \sim_s n$. The desired  $s'$, $m'$, $n'$ can be obtained as follows: think of $s$ as the multiset of arrows from $\quiver$, where the multiplicity of an arrow  $a$ is $s(a)$. Pairing off  the   new sources $\quiver'_0\setminus \quiver_0$ with the sinks in $\quiver$ not covered by $m$ and adding the corresponding arrows to $m$ we get a perfect matching $m'$ of $\quiver'$ with $\pi(m')=m$. Next do the same for $n$, with the extra condition that if none of $n$ and $m$ covers a sink in $\quiver$, then in $n'$ it is connected with the same new source as in $m'$. 
Let $t\in\mn_0^{\quiver'_1}$ be the multiset of arrows obtained  from $s$ by adding once each of the arrows $\quiver'_1\setminus \quiver_1$ occuring in $m'$ or $n'$. For a vertex $v\in\quiver'_1$ set 
$\deg_t(v):=\sum_{v\in\{c^-,c^+\}} |t(c)|$. Observe that $s-m$ and $s-n$ belong to $\semigr(\osmpoly)_{k-1}$, hence $\deg_{s-m}(w)\le k-1$ and $\deg_{s-n}(w)\le k-1$ for any vertex $w$. So if $w$ is a sink  not covered by $m$ or $n$, then $\deg_s(w)$ agrees with $\deg_{s-m}(w)$ or $\deg_{s-n}(w)$, 
thus $\deg_s(w)\le k-1$, and hence $\deg_t(w)\le k$. For the remaining sinks we have $\deg_t(w)=\deg_s(w)\le k$ as well, moreover,  
$\deg_t(v)=k$ for the  sources $v\in\quiver_0\setminus \quiver'_0$, whereas $\deg_t(v)\le 2$ for the new sources $v\in \quiver'_0\setminus \quiver_0$. 
Consequently successively adding further new arrows from $\quiver'_1\setminus \quiver_1$ to $t$ we obtain $s'\ge t$ with $\deg_{s'}(v)=k$ for all $v\in\quiver'_0$.  Moreover, 
$m'\le t\le s'$, $n'\le t\le s'$, and $\pi(s')=s$, so we are done. 
\end{proof}

%%%%%%%%%%%%%%%%%%%%%%%%%%%%%%%%

\begin{center}Acknowledgement\end{center} 

We thank Bernd Sturmfels for bringing \cite{yamaguchi-ogawa-takemura} to our attention. 

%%%%%%%%%%%%%%%%%%%%%%%%%%%%%%%%%%%

%%%%%%%%%%%%%%%%%%%%%%%%%%%%%%%%%%%

\end{document}